\documentclass[12pt]{amsart}
\subjclass[2010]{05E40, 13C05, 16W50}
\keywords{artinian modules, cone of hilbert functions, graded rings and modules, lex-segment ideals}

\usepackage{amsthm}
\usepackage{amsmath}
\usepackage{amssymb}
\usepackage{amsfonts}
\usepackage{mathtools}
\usepackage{mathdots}
\usepackage[arrow,matrix,curve,all]{xy}

\newtheorem{thm}{Theorem}[section]
\newtheorem{prop}[thm]{Proposition}
\newtheorem{lemma}[thm]{Lemma}
\newtheorem{cor}[thm]{Corollary}
\newtheorem{Def}{Definition}[section]
\theoremstyle{definition}
\newtheorem*{rem}{Remark}

\theoremstyle{definition}
\newtheorem{exa}{Example}[section]
\theoremstyle{remark}

\newcommand{\Hc}{$\mathbb{H}$}
\newcommand{\Ex}{\text{Ex}}

\usepackage{color}
\usepackage{xcolor}
\definecolor{corn}{cmyk}{0,0,1,0} \definecolor{beans}{cmyk}{.4,0.1,1,0} \definecolor{plum}{cmyk}{.6,1,0,0} \definecolor{blueberry}{cmyk}{1,.6,0,0} \definecolor{cranberry}{cmyk}{0,1,.5,0} \definecolor{apricot}{cmyk}{0,0.25,1,0} \definecolor{eggplant}{cmyk}{.25,1,0,0} \definecolor{lemon}{cmyk}{.1,0,1,0} \definecolor{potato}{cmyk}{.0005,.55,1,0} \colorlet{llemon}{lemon!60}

\newcommand{\bw}[1]{{#1}} \newcommand{\col}[1]{}                  

\definecolor{gr}{cmyk}{.25,.25,.25,.25}
\definecolor{lgr}{cmyk}{.1,.1,.1,.1}
\definecolor{dgr}{cmyk}{.55,.55,.55,.55}

\usepackage{graphicx}
\usepackage{tikz}
\usetikzlibrary{arrows,calc}
\usetikzlibrary{shapes,shapes.multipart}
\usepackage{enumitem}

\usepackage{url}

\begin{document}

\title{The Cone of Hilbert Functions in the non-standard graded Case}

\author{Daniel Brinkmann}
\address{Daniel Brinkmann\\
Institut f\"ur Mathematik\\
Universit\"at Osnabr\"uck}
\email{dabrinkm@uos.de}

\author{Marianne Merz}
\address{Marianne Merz\\
Institut f\"ur Mathematik\\
Freie Universit\"at Berlin}
\email{mamerz@math.fu-berlin.de}

\maketitle
\begin{abstract}
We describe the cone of Hilbert functions of artinian graded modules finitely generated in degree 0 over the polynomial ring $R=k[x,y]$ with the non-standard grading $\deg(x)=1$ and $\deg(y)=n$, where $n$ is any natural number.
\end{abstract}


\section{Introduction}
In 2006 Mats Boij and Jonas S\"oderberg conjectured a beautiful structure theorem on the cone of Betti tables of graded modules over the polynomial ring $R=k[x_1,\dots,x_n]$ with standard grading $\deg(x_i)=1$. They described the cone in terms of the extremal rays and gave an algorithm to decompose any finitely generated Cohen-Macaulay module in a positive rational combination of some tables generating the extremal rays, calling these tables \emph{pure}. The existence of pure Betti tables were proved in char$(k)=0$ by Eisenbud, Fl{\o}ystad and Schreyer in 2007. The full conjecture was subsequently proven in arbitrary characteristic by Eisenbud and Schreyer by providing a connection between Betti tables of modules over $R$ and cohomology tables of special vector bundles over $\mathbb{P}^{n-1}$, calling these vector bundles \emph{supernatural}. There is a recent comprehensive survey by Gunnar Fl{\o}ystad \cite{Fl}.\\

A natural question is what happens in the non-standard graded case. A research group at a summer-school in Snowbird, UT in 2010 \cite{Ba} investigated the case $R=k[x,y]$ with $\deg (x)=1$ and $\deg (y)=2$ and they realized that even in this very special case of non-standard grading it was not possible to describe the extremal rays by similar easy to define "\emph{pure}" tables.\\
Thus, the next step is to look at a coarser invariant of graded modules: the Hilbert function. In the standard graded case, Mats Boij and Greg Smith described the cone of Hilbert functions of modules of dimension $d$ finitely generated in degree 0, which coincides with the Hilbert polynomial in degrees larger than a fixed $a$ by the extremal rays as well as by the supporting hyperplanes, see \cite{Bo} and \cite{BS}.\\
In the present article, we look at a similar cone for a simple non-standard graded case. Indeed, we consider artinian graded modules generated in degree 0 over $R=k[x,y]$, where $\deg(x)=1$ and $\deg(y)=n$ for some $n \in \mathbb{N}$. We can specify the extremal rays by a recursive structure (see theorem (\ref{extr})), which provides us an algorithm to decompose any h-vector of the mentioned modules in a positive rational combination of generators of the extremal rays.

\section{Describing the Cone}

Let $R=k[x,y]$ be the graded ring with\,$\deg(x)=1$\,and\,$\deg(y)=n, \,n \geq 1$, where $k$ is a field of characteristic zero.
For an $\mathbb{N}$-graded $R$-module $M=\bigoplus_{i\geq 0} M_i$ the Hilbert function $h_M:\mathbb{N} \rightarrow \mathbb{N}$ is defined by $h_M(i):= \dim _kM_i.$ As we are only looking at artinian modules, the Hilbert function has only finitely many nontrivial values and it therefore coincides with the h-vector of $M$. For more details about Hilbert functions and h-vectors see \cite{Br}.

If we allow generators of our module to be in arbitrary degrees, there is no restriction on the possible h-vectors, in fact they would cover the whole positive orthant, so we concentrate on modules generated in degree 0.

The set of h-vectors, more generally Hilbert functions, naturally forms a semigroup: $h_{M\oplus N}(i)=h_M(i)+h_N(i)$. Therefore it makes sense to look at the cone of h-vectors, denoted by
 $$\mathbb{H} := \mathrm{cone}\{\text{h-vector of artinian }\, R\!-\!\text{modules gen. in} \deg \text{ }0\}\subseteq \bigoplus_{j \in \mathbb{N}} \mathbb{Q}.$$

We want to make our h-vectors live in a finite dimensional vectorspace. Therefore we usually work with limited degrees: \Hc$(d):= \mathbb{H}\cap \mathbb{Q}^{d+1}$. We freely identify $(h_0,\ldots,h_e,0,\ldots)$ with $(h_0,\ldots,h_e)$.

Because of the additivity of the Hilbert function,  we can write any h-vector $h \in \mathbb{H}$ of a module as a sum of h-vectors of $R$-algebras, which can be identified with quotients $R/I$ with $I$ an homogeneous ideal.
In the sequel we need the notion of lex-segment ideals, which uses the lexicographic order.

\begin{Def} The \emph{lexicographic} order of monomials $x^{a_1}y^{b_1}$, $x^{a_2}y^{b_2}\in k[x,y]$ (with arbitrary $\mathbb{N}$-grading) is defined as $$x^{a_1}y^{b_1}<_{lex}x^{a_2}y^{b_2}\;:\Leftrightarrow \;a_1>a_2\;\text{ or }\;a_1=a_2\text{ and }b_1>b_2.$$
\end{Def}

\begin{Def} Let $I$ be a monomial ideal in $R=k[x,y]$ with $\deg(x)=1$\,and\, $\deg(y)=n,\;\; I_d$ the group of homogeneous elements of degree $d$ in $I$. \\  We call $I$ a \emph{lex-segment ideal} iff for every monomial $x^ay^b\in I_d$ the monomials $x^{a+n}y^{b-1},\ldots,x^{d-n}y$, $x^d$ belong to $I_d$.
\end{Def}

\pagebreak

The following theorem due to G. Dalzotto and E. Sbarra states that $\mathbb{H}$ is generated by the h-vectors of the $R$-algebras $R/I$, where $I$ is a lex-segment ideal. 

\begin{thm}[\cite{DS}, Theorem 4.16]\label{Mac}
Let $R=k[x,y]$, $I$ be a homogeneous ideal in $R$ with $\deg(x)=1$ and $\deg(y)=n$, $n\in\mathbb{N}$. There exists a unique lex-segment ideal $L$ such that $h_{R/I}(t) = h_{R/L}(t)$ for any $t\in\mathbb{N}.$
\end{thm}

The lex-segment ideals are monomial ideals and there is a very nice and convenient way of illustrating them as staircases in the bivariate case, a more general description can be found in \cite{St}.

\begin{exa}\label{bsp1}
Let $\deg(y)=2$ and $I=\langle x^6,x^2y,xy^2,y^3 \rangle$. Then the lattice points in the non-shaded area form a $k$-basis of $R/I$.

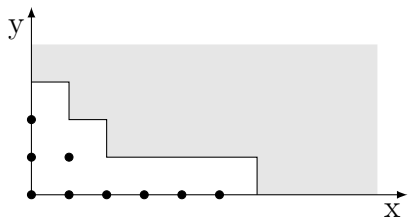
\unitlength1cm
\begin{picture}(5,3.2)
\put(1.5,0){\begin{tikzpicture}
\fill [gray!20] (4.6,0)--(3,0)--(3,.5)--(1,.5)--(1,1)--(.5,1)--(.5,1.5)--(0,1.5)--(0,2)--(4.6,2)--(4.6,0);
\draw [-latex] (0,0)--(5,0);
\draw [-latex] (0,0)--(0,2.5);
\draw (4.8,-.2) node{x} (-.2,2.2) node{y};
\draw (3,0)--(3,.5)--(1,.5)--(1,1)--(.5,1)--(.5,1.5)--(0,1.5);
\filldraw[black](0,0) circle (1.5pt) (0,.5) circle (1.5pt) (0,1) circle (1.5pt) (.5,0) circle (1.5pt) (.5,.5) circle (1.5pt)(1,0) circle (1.5pt)(1.5,0) circle (1.5pt)(2,0) circle (1.5pt)(2.5,0) circle (1.5pt);
\end{tikzpicture}
}
\end{picture}

Drawing for every generator a box marked with the corresponding degree, we get a more simple diagram:

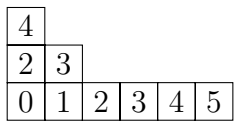
\begin{picture}(5,2.3)
\put(2,0.3){\begin{tikzpicture}
\draw (0,0)rectangle(.5,.5) (.5,0)rectangle(1,.5) (1,0)rectangle(1.5,.5) (1.5,0)rectangle(2,.5) (2,0)rectangle(2.5,.5) (2.5,0)rectangle(3,.5) (0,.5)rectangle(.5,1) (.5,.5)rectangle(1,1) (0,1)rectangle(.5,1.5);
\draw (.25,.25) node{0} (.75,.25) node{1} (1.25,.25) node{2} (1.75,.25) node{3} (2.25,.25) node{4} (2.75,.25) node{5} (.25,.75) node{2} (.75,.75) node{3} (.25,1.25) node{4};
\end{tikzpicture}
}
\end{picture}

The h-vector of this module is $h=(1,1,2,2,2,1)$, but there exist other monomial ideals with the same h-vector for the quotient:

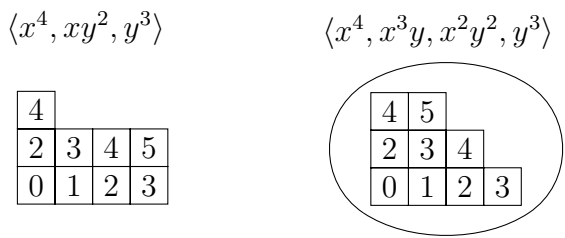
\begin{picture}(5,3.7)
\put(2,2.75){$\langle x^4,xy^2,y^3 \rangle$}
\put(2,.5){ \begin{tikzpicture}
\draw (0,0)rectangle(.5,.5) (.5,0)rectangle(1,.5) (1,0)rectangle(1.5,.5) (1.5,0)rectangle(2,.5) 
(0,.5)rectangle(.5,1) (.5,.5)rectangle(1,1) (1,.5)rectangle(1.5,1) (1.5,.5)rectangle(2,1) 
(0,1)rectangle(.5,1.5);
\draw (.25,.25) node{0} (.75,.25) node{1} (1.25,.25) node{2} (1.75,.25) node{3}
 (.25,.75) node{2} (.75,.75) node{3} (1.25,.75) node{4} (1.75,.75) node{5} 
 (.25,1.25) node{4};
\end{tikzpicture}}

\put(6.2,2.7){$\langle x^4,x^3y,x^2y^2,y^3 \rangle$}
\put(6,0){
\begin{tikzpicture}
\draw (0,0)rectangle(.5,.5) (.5,0)rectangle(1,.5) (1,0)rectangle(1.5,.5) (1.5,0)rectangle(2,.5) 
(0,.5)rectangle(.5,1) (.5,.5)rectangle(1,1) (1,.5)rectangle(1.5,1)  
(0,1)rectangle(.5,1.5) (.5,1)rectangle(1,1.5);
\draw (.25,.25) node{0} (.75,.25) node{1} (1.25,.25) node{2} (1.75,.25) node{3}
 (.25,.75) node{2} (.75,.75) node{3} (1.25,.75) node{4}  
 (.25,1.25) node{4} (.75,1.25) node{5};
\draw (-.1,-.1) .. controls (.5,-.5) and (1.5,-.5)..(2.1,-.1) ..controls (2.7,.3) and (2.7,1.2)..(2.1,1.6);
\draw (-.1,-.1) .. controls (-.7,.3) and (-.7,1.2)..(-.1,1.6) ..controls (.5,2) and (1.5,2)..(2.1,1.6);
\end{tikzpicture}}

\end{picture}

The encircled staircase corresponds to the lex-segment ideal with respect to $h$; just stack the boxes as far left as possible. In the sequel we will identify frequently an h-vector with the corresponding staircase diagram of the lex-segment ideal.
\end{exa}

\pagebreak

Given the h-vector $h$ as a sum of h-vectors of $R$-algebras we attach to this decomposition a three-dimensional diagram with boxes nested in a corner as follows:

\begin{itemize}[leftmargin=2em]
  \item[(1)] Tilt the corresponding staircases forward so they are lying flat on the floor.
	\item[(2)] Blow them up to cubes of height one.
	\item[(3)] Stack the resulting box diagrams with respect to their degrees. 
	\item[(4)] Some of these boxes may overlap, then drop these boxes down to the next lower box of the same degree.
\end{itemize}

We call this stack an $h$-$diagram$ corresponding to the h-vector $h$. Note that these diagrams depend on the decomposition of $h$, only the total number of boxes in each degree is fixed by $h$.
We demonstrate the construction in an example:

\begin{exa}
Let $n=3$ and $h=(3,2,1,4,1,0,2,1)=$\\ $=\;(1,1,1,2,0,0,0,0)\;+\;(1,1,0,1,1,0,1,1)\;+\;(1,0,0,1,0,0,1,0)$\\
The corresponding staircases look like this:

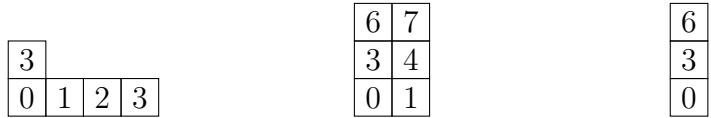
\unitlength1cm
\begin{picture}(11,2)
\put(1,.2){
\begin{tikzpicture}
\draw (0,0)rectangle(.5,.5);
\draw (.5,0)rectangle(1,.5);
\draw (1,0)rectangle(1.5,.5);
\draw (1.5,0)rectangle(2,.5);
\draw (0,.5)rectangle(.5,1);
\draw (.25,.25) node{0} (.75,.25) node{1} (1.25,.25) node{2} (1.75,.25) node{3}
      (.25,.75) node{3};      
\end{tikzpicture}
}

\put(5.6,.2){
\begin{tikzpicture}
\draw (0,0)rectangle(.5,.5);
\draw (.5,0)rectangle(1,.5);
\draw (0,.5)rectangle(.5,1);
\draw (.5,.5)rectangle(1,1);
\draw (0,1)rectangle(.5,1.5);
\draw (.5,1)rectangle(1,1.5);
\draw (.25,.25) node{0} (.75,.25) node{1} 
      (.25,.75) node{3} (.75,.75) node{4}
      (.25,1.25) node{6} (.75,1.25) node{7};      
\end{tikzpicture}
}

\put(9.8,.2){
\begin{tikzpicture}
\draw (0,0)rectangle(.5,.5);
\draw (0,.5)rectangle(.5,1);
\draw (0,1)rectangle(.5,1.5); 
\draw (.25,.25) node{0} (.25,.75) node{3} (.25,1.25) node{6};     
\end{tikzpicture}
}
\end{picture}

Turning the staircases down and blowing them up gives:

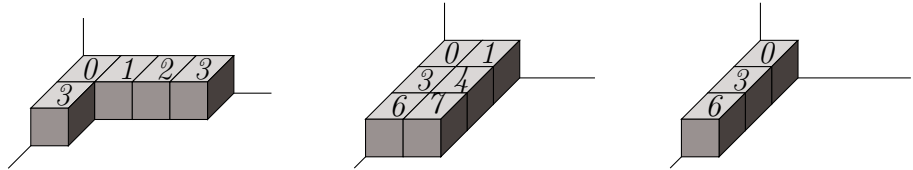
\begin{picture}(11,2.3)
\put(.8,0){
\begin{tikzpicture}
 \draw (0,0)--(2.5,0);
 \draw (0,0)--(0,1);
 \draw (0,0)--(-1,-1);
 \filldraw [fill=lgr!70,draw=black] (0,.5)--(-.5*.7,.5-.5*.7)--(.5-.5*.7,.5-.5*.7)--(.5,.5)--(0,.5);
 \filldraw [fill=lgr!70,draw=black] (.5,.5)--(.5-.5*.7,.5-.5*.7)--(1-.5*.7,.5-.5*.7)--(1,.5)--(.5,.5);
 \filldraw [fill=lgr!70,draw=black] (1,.5)--(1-.5*.7,.5-.5*.7)--(1.5-.5*.7,.5-.5*.7)--(1.5,.5)--(1,.5);
 \filldraw [fill=lgr!70,draw=black] (1.5,.5)--(1.5-.5*.7,.5-.5*.7)--(2-.5*.7,.5-.5*.7)--(2,.5)--(1.5,.5);
 \filldraw [fill=lgr!70,draw=black] (-.5*.7,.5-.5*.7)--(.5-.5*.7,.5-.5*.7)--(.5-.7,.5-.7)--(-.7,.5-.7)--(-.5*.7,.5-.5*.7);
 
 \filldraw [fill=dgr, draw=black] (.5-.5*.7,.5-.5*.7)--(.5-.7,.5-.7)--(.5-.7,-.7)--(.5-.5*.7,-.5*.7)--(.5-.5*.7,.5-.5*.7);
 \filldraw [fill=dgr, draw=black] (2,.5)--(2-.5*.7,.5-.5*.7)--(2-.5*.7,-.5*.7)--(2,0)--(2,.5);
 
 \filldraw [fill=gr, draw=black] (-.7,.5-.7) rectangle (.5-.7,-.7);
 \filldraw [fill=gr, draw=black] (.5-.5*.7,.5-.5*.7) rectangle (1-.5*.7,-.5*.7);
 \filldraw [fill=gr, draw=black] (1-.5*.7,.5-.5*.7) rectangle (1.5-.5*.7,-.5*.7);
 \filldraw [fill=gr, draw=black] (1.5-.5*.7,.5-.5*.7) rectangle (2-.5*.7,-.5*.7);
 
 \draw (.45-.5*.7,.68-.5*.7) node{\emph{0}};
 \draw (.95-.5*.7,.68-.5*.7) node{\emph{1}};
 \draw (1.45-.5*.7,.68-.5*.7) node{\emph{2}};
 \draw (1.95-.5*.7,.68-.5*.7) node{\emph{3}};
 \draw (.45-1*.7,.68-1*.7) node{\emph{3}};
\end{tikzpicture}
}
 \put(5.4,0){
\begin{tikzpicture}
 \draw (0,0)--(2,0);
 \draw (0,0)--(0,1);
 \draw (0,0)--(-1.2,-1.2);
 \filldraw [fill=lgr!70,draw=black] (0,.5)--(-.5*.7,.5-.5*.7)--(.5-.5*.7,.5-.5*.7)--(.5,.5)--(0,.5);
 \filldraw [fill=lgr!70,draw=black] (.5,.5)--(.5-.5*.7,.5-.5*.7)--(1-.5*.7,.5-.5*.7)--(1,.5)--(.5,.5);
 \filldraw [fill=lgr!70,draw=black] (-.5*.7,.5-.5*.7)--(-1*.7,.5-1*.7)--(.5-1*.7,.5-1*.7)--(.5-.5*.7,.5-.5*.7)--(-.5*.7,.5-.5*.7);
 \filldraw [fill=lgr!70,draw=black] (.5-.5*.7,.5-.5*.7)--(.5-1*.7,.5-1*.7)--(1-1*.7,.5-1*.7)--(1-.5*.7,.5-.5*.7)--(.5-.5*.7,.5-.5*.7);
 \filldraw [fill=lgr!70,draw=black] (-1*.7,.5-1*.7)--(-1.5*.7,.5-1.5*.7)--(.5-1.5*.7,.5-1.5*.7)--(.5-1*.7,.5-1*.7)--(-1*.7,.5-1*.7);
 \filldraw [fill=lgr!70,draw=black] (.5-1*.7,.5-1*.7)--(.5-1.5*.7,.5-1.5*.7)--(1-1.5*.7,.5-1.5*.7)--(1-1*.7,.5-1*.7)--(.5-1*.7,.5-1*.7);
 
 \filldraw [fill=dgr, draw=black] (1,.5)--(1-.5*.7,.5-.5*.7)--(1-.5*.7,-.5*.7)--(1,0)--(1,.5);
 \filldraw [fill=dgr, draw=black] (1-.5*.7,.5-.5*.7)--(1-1*.7,.5-1*.7)--(1-1*.7,-1*.7)--(1-.5*.7,-.5*.7)--(1-.5*.7,.5-.5*.7);
 \filldraw [fill=dgr, draw=black] (1-1*.7,.5-1*.7)--(1-1.5*.7,.5-1.5*.7)--(1-1.5*.7,-1.5*.7)--(1-1*.7,-1*.7)--(1-1*.7,.5-1*.7);
 
 \filldraw [fill=gr, draw=black] (-1.5*.7,.5-1.5*.7)rectangle (.5-1.5*.7,-1.5*.7);
 \filldraw [fill=gr, draw=black] (.5-1.5*.7,.5-1.5*.7)rectangle (1-1.5*.7,-1.5*.7);
 
 \draw (.45-.5*.7,.68-.5*.7) node{\emph{0}};
 \draw (.95-.5*.7,.68-.5*.7) node{\emph{1}};
 \draw (.45-1*.7,.68-1*.7) node{\emph{3}};
 \draw (.95-1*.7,.68-1*.7) node{\emph{4}};
 \draw (.45-1.5*.7,.68-1.5*.7) node{\emph{6}};
 \draw (.95-1.5*.7,.68-1.5*.7) node{\emph{7}};
\end{tikzpicture}
}
 \put(9.6,0){
\begin{tikzpicture}
 \draw (0,0)--(2,0);
 \draw (0,0)--(0,1);
 \draw (0,0)--(-1.2,-1.2);
 \filldraw [fill=lgr!70,draw=black] (0,.5)--(-.5*.7,.5-.5*.7)--(.5-.5*.7,.5-.5*.7)--(.5,.5)--(0,.5);
 \filldraw [fill=lgr!70,draw=black] (-.5*.7,.5-.5*.7)--(-1*.7,.5-1*.7)--(.5-1*.7,.5-1*.7)--(.5-.5*.7,.5-.5*.7)--(-.5*.7,.5-.5*.7);
 \filldraw [fill=lgr!70,draw=black] (-1*.7,.5-1*.7)--(-1.5*.7,.5-1.5*.7)--(.5-1.5*.7,.5-1.5*.7)--(.5-1*.7,.5-1*.7)--(-1*.7,.5-1*.7);
 
 \filldraw [fill=dgr, draw=black] (.5,.5)--(.5-.5*.7,.5-.5*.7)--(.5-.5*.7,-.5*.7)--(.5,0)--(.5,.5);
 \filldraw [fill=dgr, draw=black] (.5-.5*.7,.5-.5*.7)--(.5-1*.7,.5-1*.7)--(.5-1*.7,-1*.7)--(.5-.5*.7,-.5*.7)--(.5-.5*.7,.5-.5*.7);
 \filldraw [fill=dgr, draw=black] (.5-1*.7,.5-1*.7)--(.5-1.5*.7,.5-1.5*.7)--(.5-1.5*.7,-1.5*.7)--(.5-1*.7,-1*.7)--(.5-1*.7,.5-1*.7);
 \filldraw [fill=gr, draw=black] (-1.5*.7,.5-1.5*.7)rectangle (.5-1.5*.7,-1.5*.7);
 
 \draw (.45-.5*.7,.68-.5*.7) node{\emph{0}};
 \draw (.45-1*.7,.68-1*.7) node{\emph{3}};
 \draw (.45-1.5*.7,.68-1.5*.7) node{\emph{6}};
 
\end{tikzpicture}
}
\end{picture}

Stacking these box diagrams with respect to their degrees we get one big stack:

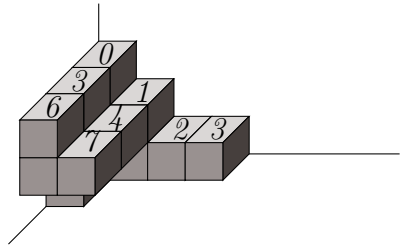
\begin{picture}(8,3.3)
\put(2.6,0){
\begin{tikzpicture}
 \draw (0,0)--(4,0);
 \draw (0,0)--(0,2);
 \draw (0,0)--(-1.2,-1.2);
\filldraw [fill=gr,draw=black] (.5-.5*.7,.5-.5*.7)rectangle ( 1-.5*.7,-.5*.7);
\filldraw [fill=lgr!70, draw=black]  (1,.5)--(1.5,.5)--(1.5-.5*.7,.5-.5*.7)--(1-.5*.7,.5-.5*.7)--(1,.5);
\filldraw [fill=dgr, draw=black] (1.5,.5)--(1.5,0)--(1.5-.5*.7,-.5*.7)--(1.5-.5*.7,.5-.5*.7)--(1.5,.5);
\filldraw [fill=gr, draw=black] (1-.5*.7,-.5*.7)rectangle(1.5-.5*.7,.5-.5*.7);
\draw  (1.45-.5*.7,.68-.5*.7) node{\emph{2}};
\filldraw [fill=lgr!70, draw=black]  (1.5,.5)--(2,.5)--(2-.5*.7,.5-.5*.7)--(1.5-.5*.7,.5-.5*.7)--(1.5,.5);
\filldraw [fill=dgr, draw=black]  (2,.5)--(2,0)--(2-.5*.7,-.5*.7)--(2-.5*.7,.5-.5*.7)--(2,.5);
\filldraw [fill=dgr, draw=black] (.5-1*.7,-1*.7+.5)--(.5-.5*.7,-.5*.7+.5)--(.5-.5*.7,-.5*.7)--(.5-1*.7,-1*.7)--(.5-1*.7,-1*.7+.5);
\filldraw [fill=gr, draw=black] (1.5-.5*.7,.5-.5*.7)rectangle(2-.5*.7,-.5*.7);
\filldraw [fill=gr, draw=black] (-1*.7,.5-1*.7)rectangle(.5-1*.7,-1*.7);
\draw (1.95-.5*.7,.68-.5*.7) node{\emph{3}};

\filldraw [fill=lgr!70, draw=black] (.5,1)--(.5-.5*.7,1-.5*.7)--(1-.5*.7,1-.5*.7)--(1,1)--(.5,1);
\filldraw [fill=dgr, draw=black] (1,1)--(1-.5*.7,1-.5*.7)--(1-.5*.7,.5-.5*.7)--(1,.5)--(1,1);
\filldraw [fill=gr, draw=black] (1-.5*.7,1-.5*.7)rectangle(.5-.5*.7,-.5*.7+.5);
\draw  (.95-.5*.7,1.18-.5*.7) node{\emph{1}};
\filldraw [fill=lgr!70, draw=black] (.5-1*.7,1-1*.7)-- (.5-.5*.7,1-.5*.7)-- (1-.5*.7,1-.5*.7)-- (1-1*.7,1-1*.7)-- (.5-1*.7,1-1*.7);
\filldraw [fill=dgr, draw=black] (1-.5*.7,1-.5*.7)--(1-1*.7,1-1*.7)--(1-1*.7,.5-1*.7)--(1-.5*.7,.5-.5*.7)--(1-.5*.7,1-.5*.7);
\draw (.95-1*.7,1.18-1*.7) node{\emph{4}};
\filldraw [fill=gr, draw=black]  (-1.5*.7,1-1.5*.7)rectangle (.5-1.5*.7,.5-1.5*.7) ;
\filldraw  [fill=lgr!70, draw=black](.5-1.5*.7,1-1.5*.7)-- (.5-1*.7,1-1*.7)-- (1-1*.7,1-1*.7)-- (1-1.5*.7,1-1.5*.7)-- (.5-1.5*.7,1-1.5*.7);
\filldraw [fill=dgr, draw=black]   (1-1*.7,1-1*.7)--(1-1.5*.7,1-1.5*.7)--(1-1.5*.7,.5-1.5*.7)--(1-1*.7,.5-1*.7)--(1-1*.7,1-1*.7);
\filldraw [fill=gr, draw=black] (.5-1.5*.7,.5-1.5*.7)rectangle(1-1.5*.7,1-1.5*.7);
\draw  (.95-1.5*.7,1.18-1.5*.7) node{\emph{7}};

\filldraw [fill=lgr!70,draw=black] (0,1.5)--(-.5*.7,1.5-.5*.7)--(.5-.5*.7,-.5*.7+1.5)--(.5,1.5)--(0,1.5);
\filldraw [fill=dgr, draw=black] (.5-.5*.7,-.5*.7+1.5)--(.5,1.5)--(.5,1)--(.5-.5*.7,-.5*.7+1)--(.5-.5*.7,-.5*.7+1.5);
\draw (.45-.5*.7,1.68-.5*.7) node{\emph{0}};
\filldraw [fill=lgr!70, draw=black]  (-.5*.7,1.5-.5*.7)--(.5-.5*.7,-.5*.7+1.5)--(.5-1*.7,-1*.7+1.5)--(-1*.7,-1*.7+1.5)--(-.5*.7,1.5-.5*.7);
\filldraw [fill=dgr, draw=black] (.5-.5*.7,-.5*.7+1.5)--(.5-1*.7,-1*.7+1.5)--(.5-1*.7,-1*.7+1)--(.5-.5*.7,-.5*.7+1)--(.5-.5*.7,-.5*.7+1.5);
\filldraw [fill=gr, draw=black] (-1*.7,-1*.7+1.5)rectangle(.5-1*.7,-1*.7+1);
\draw (.45-1*.7,1.68-1*.7) node{\emph{3}};
\filldraw [fill=lgr!70, draw=black] (-1*.7,1.5-1*.7)--(.5-1*.7,-1*.7+1.5)--(.5-1.5*.7,-1.5*.7+1.5)--(-1.5*.7,-1.5*.7+1.5)--(-1*.7,-1*.7+1.5);
\filldraw [fill=dgr, draw=black] (.5-1*.7,-1*.7+1.5)--(.5-1.5*.7,-1.5*.7+1.5)--(.5-1.5*.7,-1.5*.7+1)--(.5-1*.7,-1*.7+1)-- (.5-1*.7,-1*.7+1.5);   
\filldraw [fill=gr, draw=black]  (-1.5*.7,1.5-1.5*.7)rectangle (.5-1.5*.7,1-1.5*.7) ;
\draw  (.45-1.5*.7,1.68-1.5*.7) node{\emph{6}};
\end{tikzpicture}
}
\end{picture}


Now some of these boxes are not grounded so we let them drop:

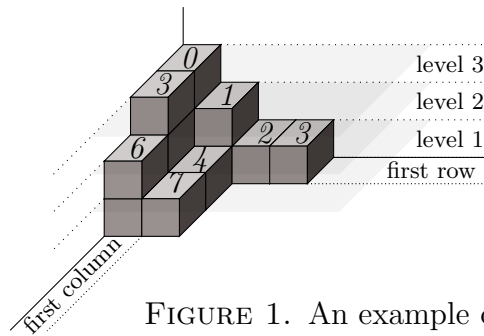
\begin{figure}[h]
\begin{picture}(12,3.1)
\put(0.4,-1.2){
\begin{tikzpicture}
 \draw (0,0)--(4,0);
 \draw (0,0)--(0,2);
 \draw (0,0)--(-2.3,-2.3);
\filldraw [fill=gr,draw=black] (.5-1.5*.7,.5-1.5*.7)rectangle ( -1.5*.7,-1.5*.7);
\filldraw  [fill=lgr!70, draw=black](.5-1.5*.7,.5-1.5*.7)-- (.5-1*.7,.5-1*.7)-- (1-1*.7,.5-1*.7)-- (1-1.5*.7,.5-1.5*.7)-- (.5-1.5*.7,.5-1.5*.7);
\filldraw [fill=dgr, draw=black]   (1-1*.7,.5-1*.7)-- (1-1.5*.7,.5-1.5*.7)--(1-1.5*.7,-1.5*.7)--(1-1*.7,-1*.7)--(1-1*.7,.5-1*.7);
\filldraw [fill=gr, draw=black] (.5-1.5*.7,-1.5*.7)rectangle(1-1.5*.7,.5-1.5*.7);
\draw  (.95-1.5*.7,.68-1.5*.7) node{\emph{7}};
\filldraw [fill=lgr!70, draw=black] (.5-1*.7,.5-1*.7)-- (.5-.5*.7,.5-.5*.7)-- (1-.5*.7,.5-.5*.7)-- (1-1*.7,.5-1*.7)-- (.5-1*.7,.5-1*.7);
\filldraw [fill=dgr, draw=black] (1-.5*.7,.5-.5*.7)--(1-1*.7,.5-1*.7)--(1-1*.7,-1*.7)--(1-.5*.7,-.5*.7)--(1-.5*.7,.5-.5*.7);
\draw (.95-1*.7,.68-1*.7) node{\emph{4}};
\filldraw [fill=lgr!70, draw=black]  (1,.5)--(1.5,.5)--(1.5-.5*.7,.5-.5*.7)--(1-.5*.7,.5-.5*.7)--(1,.5);
\filldraw [fill=dgr, draw=black] (1.5,.5)--(1.5,0)--(1.5-.5*.7,-.5*.7)--(1.5-.5*.7,.5-.5*.7)--(1.5,.5);
\filldraw [fill=gr, draw=black] (1-.5*.7,-.5*.7)rectangle(1.5-.5*.7,.5-.5*.7);
\draw  (1.45-.5*.7,.68-.5*.7) node{\emph{2}};
\filldraw [fill=lgr!70, draw=black]  (1.5,.5)--(2,.5)--(2-.5*.7,.5-.5*.7)--(1.5-.5*.7,.5-.5*.7)--(1.5,.5);
\filldraw [fill=dgr, draw=black]  (2,.5)--(2,0)--(2-.5*.7,-.5*.7)--(2-.5*.7,.5-.5*.7)--(2,.5);
\filldraw [fill=gr, draw=black] (1.5-.5*.7,.5-.5*.7)rectangle(2-.5*.7,-.5*.7);
\draw (1.95-.5*.7,.68-.5*.7) node{\emph{3}};
\fill[opacity=0.05](0,.5)--(-1.75*.7,.5-1.75*.7)--(3.3-1.75*.7,.5-1.75*.7)--(3.3,.5);
\draw [dotted] (2,.5)--(4,.5);
\draw [dotted] (-1.5*.7,.5-1.5*.7)--  (-2.5*.7,.5-2.5*.7);
\draw (3.55,.25)node{\scriptsize{level 1}};

\filldraw [fill=lgr!70, draw=black] (-1*.7,-1*.7+1)--(.5-1*.7,-1*.7+1)--(.5-1.5*.7,-1.5*.7+1)--(-1.5*.7,-1.5*.7+1)--(-1*.7,-1*.7+1);
\filldraw [fill=dgr, draw=black] (.5-1*.7,-1*.7+1)--(.5-1.5*.7,-1.5*.7+1)--(.5-1.5*.7,-1.5*.7+.5)--(.5-1*.7,-1*.7+.5)-- (.5-1*.7,-1*.7+1);   
\filldraw [fill=gr, draw=black]  (-1.5*.7,1-1.5*.7)rectangle (.5-1.5*.7,.5-1.5*.7) ;
\draw  (.45-1.5*.7,1.18-1.5*.7) node{\emph{6}};
\filldraw [fill=dgr, draw=black]  (.5-1*.7,-1*.7+1)--(.5-.5*.7,-.5*.7+1)--(.5-.5*.7,-.5*.7+.5)--(.5-1*.7,-1*.7+.5)--(.5-1*.7,-1*.7+1);
\filldraw [fill=lgr!70, draw=black] (.5,1)--(.5-.5*.7,1-.5*.7)--(1-.5*.7,1-.5*.7)--(1,1)--(.5,1);
\filldraw [fill=dgr, draw=black] (1,1)--(1-.5*.7,1-.5*.7)--(1-.5*.7,.5-.5*.7)--(1,.5)--(1,1);
\filldraw [fill=gr, draw=black] (1-.5*.7,1-.5*.7)rectangle(.5-.5*.7,-.5*.7+.5);
\draw  (.95-.5*.7,1.18-.5*.7) node{\emph{1}};
\fill[opacity=0.05](0,1)--(-1.75*.7,1-1.75*.7)--(3.3-1.75*.7,1-1.75*.7)--(3.3,1);
\draw [dotted] (1,1)--(4,1);
\draw [dotted] (-1.5*.7,1-1.5*.7)--  (-2.5*.7,1-2.5*.7);
\draw (3.55,.75)node{\scriptsize{level 2}};

\filldraw [fill=lgr!70,draw=black] (0,1.5)--(-.5*.7,1.5-.5*.7)--(.5-.5*.7,-.5*.7+1.5)--(.5,1.5)--(0,1.5);
\filldraw [fill=dgr, draw=black] (.5-.5*.7,-.5*.7+1.5)--(.5,1.5)--(.5,1)--(.5-.5*.7,-.5*.7+1)--(.5-.5*.7,-.5*.7+1.5);
\draw (.45-.5*.7,1.68-.5*.7) node{\emph{0}};
\filldraw [fill=lgr!70, draw=black]  (-.5*.7,1.5-.5*.7)--(.5-.5*.7,-.5*.7+1.5)--(.5-1*.7,-1*.7+1.5)--(-1*.7,-1*.7+1.5)--(-.5*.7,1.5-.5*.7);
\filldraw [fill=dgr, draw=black] (.5-.5*.7,-.5*.7+1.5)--(.5-1*.7,-1*.7+1.5)--(.5-1*.7,-1*.7+1)--(.5-.5*.7,-.5*.7+1)--(.5-.5*.7,-.5*.7+1.5);
\filldraw [fill=gr, draw=black] (-1*.7,-1*.7+1.5)rectangle(.5-1*.7,-1*.7+1);
\draw (.45-1*.7,1.68-1*.7) node{\emph{3}};
\fill[opacity=0.05](0,1.5)--(-1.75*.7,1.5-1.75*.7)--(3.3-1.75*.7,1.5-1.75*.7)--(3.3,1.5);
\draw [dotted] (.5,1.5)--(4,1.5);
\draw [dotted] (-1*.7,1.5-1*.7)--  (-2.5*.7,1.5-2.5*.7);
\draw (3.55,1.25)node{\scriptsize{level 3}};
\draw [densely dotted] (2,0)--(4,0) (2-.5*.7,-.5*.7)--(4,-.5*.7);
\draw (3.3,-.17) node {\scriptsize{first row}}; 
\draw [densely dotted] (.5-1.5*.7,-1.5*.7)--  (.5-3.3*.7,-3.3*.7);
\draw (.2-2.4*.7,-2.4*.7) node{\rotatebox{45}{\scriptsize{first column}}};
\end{tikzpicture}
}
\end{picture}
\caption{An example of an h-diagram}
\end{figure}
\end{exa}

We denote the maximal layers of the h-diagram where no internal stairs occur by levels and count them from the bottom. Later on these levels don't have to be of integer height. We are also talking about rows in the h-diagram meaning the stack of boxes in the rows parallel to the rear wall and columns analogous to the side wall as described in the picture.

The staircases are by construction decreasing, hence the boxes in the h-diagram are also decreasing in every row and in every column. Of course, for any such decreasing box diagram we can find a module with corresponding h-vector by identifying every level with the staircase of an ideal times the height and summing up. Considering the cone of h-vectors is equivalent to allowing levels in every rational height. Therefore we get the following proposition:


\begin{prop}\label{hcone}
Let\, $\deg(y)=n,\; n \in \mathbb{N}$. An element $h=(h_0,\ldots,h_d) \in \mathbb{Q}_{\geq 0}^{d+1}$ \linebreak belongs to the cone \Hc(d) if and only if there exists a decomposition of the components $h_i=\sum_{j=1}^{s_i}h_i^j$\; with \,$h_i^j\! \geq\! 0$\;for all\, $i=0,...,d$,\, $h_i^j\!=\!0$ for $j\! >\! s_i$ \,and
\begin{itemize}
\item[(1)] $h_i^j \geq h_{i+1}^j$ \hspace{13mm} for all\; $j=1,\ldots,s_d$ and $i=n(j-1),\ldots,d$\;  and \\
\vspace{-3mm}
\item[(2)] $h_{ni+r}^j \geq h_{n(i+1)+r}^{j+1}$\;\; for all \;$i=0,...,\lfloor \frac{d}{n}\rfloor,\; j=1,...,s_{ni}$ \;and\, $r=0,...,n-1$.
\end{itemize}
We will call any such decomposition an \emph{h-diagram}.
\end{prop}

\begin{proof}

To show that every element of the cone admits such a decomposition it is enough to show it for the generators since both conditions are additive.

Let $h=(1,h_1,\ldots,h_d)$ be a generator of the cone, e.g. the $h$-vector of some $R/I$. 
By theorem (\ref{Mac}) we may assume that $I$ is monomial, hence we can look at the corresponding staircase. Every box in this staircase marked with $i$ stands for a generator of $R/I$ of degree $i$. Setting $h_i^j=1$ for having a box marked with $i$ in the $j$-th row and $h_i^j=0$ if not we get the desired decomposition. As the staircases as the visualization of ideals are always nested in the corner, saying there are no holes and the boxes are decreasing from left to right, the conditions (1) and (2) are fulfilled.

For the other direction we build an h-diagram out of the decomposition of a vector $h \in \mathbb{Q}_{\geq0}^{d+1}$ setting $h_i^j$ boxes in the $j$-th row and the $(i-(j-1)n+1)$-th column of a three-dimensional diagram. Condition (1) ensures that the rows in this diagram are decreasing and condition (2) that the columns are decreasing as well. Cutting this diagram into levels as described before we get in every level $\ell$ a staircase corresponding to an ideal $I_\ell$ blown up to the levelheight $q_\ell$.\\
 Let $q$ be a common multiple of $q_\ell$ then $$M= \bigoplus_{levels}qq_\ell\, R/I_\ell$$ is an $R$-module with h-vector $q \cdot h$ of degree $d$ and therefore $h \in \mathbb{H}(d)$.\qedhere
\end{proof}

\pagebreak
   
Next we want to list the extremal points that are the first integer points on the extremal rays. We denote by $\Ex(d)$ the extremal points of the cone \Hc$(d)$ for a fixed integer $d$.\\ 

There are some distinguished h-vectors we want to give a special notation.

 We denote by $s^d=(\underbrace{1,...,1}_{n},2,\ldots,2,3,\ldots\lfloor\frac{d}{n}\rfloor+1)$ the h-vector of length $d+1$ of an R-algebra generated in degree 0 with maximal entries
and let $s_m$ be the $m$-th coefficient of $s^d$ for nontrivial entries ignoring the upper $d$ .

For $d \in \mathbb{N}$ write $d=n\cdot m +r$ where $r \in \{0,\dots,n-1\}$. We denote by $t^d$ the h-vector of the shape
$$(\underbrace{1,\ldots,1}_{r+1},\underbrace{0,\ldots,0}_{n-r-1},1,\ldots,1,0,\ldots,0,\ldots,\underbrace{1,\ldots,1}_{r+1})$$ 
The parts $1,\ldots,1,0,\ldots,0$ occur $m$ times, therefore $t^d$ has length $d+1$.\\
Note that this is the h-vector of the staircase given by a rectangle of size \linebreak $(r+1)\times (m+1)$,  we call $t^d$  the \emph{tower} of degree $d$.
In example (\ref{notextr}) we will see that the towers for $d\equiv n-1$ modulo $n$ are decomposable.\\

To write down the extremal points we need a kind of glueing operation:

\begin{Def}
Let $d=n\cdot m+r$ with $r\in\{0,\ldots,n-1\}$ \;and $h \in \mathbb{H}(n\cdot m-r-3)$. Then we define
$$t^d*h:=t^d+(\underbrace{0,\ldots,0}_{r+1},h_0,\ldots,h_{n\cdot m-r-3},\underbrace{0,\ldots,0}_{r+2}),$$
\end{Def}

\begin{rem}
The definition of the $*$-operation is not as arbitrary as it may look. Stated in terms of staircases, this is just the procedure of taking $t^d$ and glueing $h$ on the right hand side. This is obviously still an h-vector of length $d+1$.
\end{rem}

\begin{exa}
Let $n=3, d=7=3\cdot 2+1$ and $h=(1,1)$.\\We get \;$t^d= (1,1,0,1,1,0,1,1) $\, and \;$t^d*h= (1,1,1,2,1,0,1,1).$

In the language of staircases:

\unitlength1cm
\begin{picture}(10,1.3)
\put(2,0){
\begin{tikzpicture}
\draw (0,0)rectangle (.3,.3) (.3,0)rectangle (.6,.3) (0,.3)rectangle (.3,.6) (.3,.3)rectangle (.6,.6) (0,.6)rectangle (.3,.9) (.3,.6)rectangle (.6,.9);
\draw (1,.5) node {$*$};
\end{tikzpicture}
}
\put(3.5,.3){
\begin{tikzpicture}
\draw (0,0)rectangle (.3,.3) (.3,0)rectangle (.6,.3);
\draw (1.2,.2) node {$=$};
\end{tikzpicture}
}
\put(5.5,0){
\begin{tikzpicture}
\draw (0,0)rectangle (.3,.3) (.3,0)rectangle (.6,.3) (0,.3)rectangle (.3,.6) (.3,.3)rectangle (.6,.6) (0,.6)rectangle (.3,.9) (.3,.6)rectangle (.6,.9) (.6,0)rectangle (.9,.3) (.9,0)rectangle (1.2,.3);

\end{tikzpicture}
}
\end{picture}
\end{exa}

Now we can state our main result which will be proven in section 3:

\begin{thm}\label{extr}
Let $\deg(y)=n \in \mathbb{N}$  \, and\, $d=n\cdot m+r$ with $m \geq 0$ and $ r\in \{0,\dots,n-1\}$.
 The extremal points of \;\Hc(d)\, are given by:
\begin{itemize}
\item[(0)] For $d \leq n-1: \;\;\Ex(d)=\{h=(1,\dots,1)$ of length $ \leq d+1 \}$.
\item[(1)] For $r\in\{0,\ldots,n-2\}: \;\; \Ex(d)=\Ex(d-1)\cup s^{d}\cup t^{d}\cup t^{d}*\Ex(d-2r-3)$.
\item[(2)] For $r=n-1: \;\;\Ex(d)=\Ex(d-1)\cup s^{d}$.
\end{itemize}
\end{thm}

\begin{exa}
Let $\deg(y)=2$. We look for the extremal points up to degree 6:

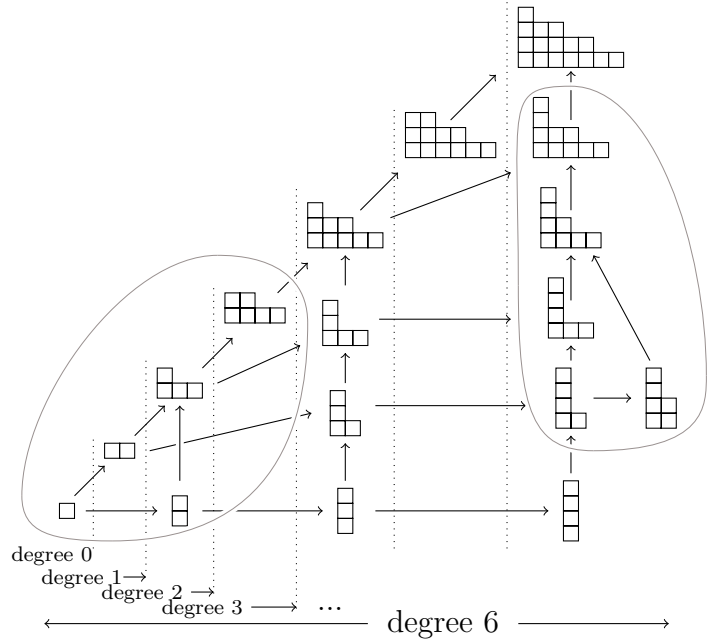
\unitlength1cm
\begin{picture}(11,9)
\put(2,0){
\begin{tikzpicture}
\draw (-.1,.4)rectangle(.1,.6) 
      (.5,1.2)rectangle(.7,1.4) (.7,1.2)rectangle(.9,1.4)
      (1.2,2)rectangle(1.4,2.2) (1.4,2)rectangle(1.6,2.2) (1.6,2)rectangle(1.8,2.2) (1.2,2.2)rectangle(1.4,2.4)                       (1.4,.3)rectangle(1.6,.5) (1.4,.5)rectangle(1.6,.7)
      (2.1,3)rectangle(2.3,3.2) (2.3,3)rectangle(2.5,3.2) (2.5,3)rectangle(2.7,3.2) (2.7,3)rectangle(2.9,3.2) 
      (2.1,3.2)rectangle(2.3,3.4) (2.3,3.2)rectangle(2.5,3.4)
      (3.2,4)rectangle(3.4,4.2)(3.4,4)rectangle(3.6,4.2) (3.6,4)rectangle(3.8,4.2)(3.8,4)rectangle(4,4.2)(4,4)rectangle(4.2,4.2)
      (3.2,4.2)rectangle(3.4,4.4)(3.4,4.2)rectangle(3.6,4.4) (3.6,4.2)rectangle(3.8,4.4)
      (3.2,4.4)rectangle(3.4,4.6) 
      
      (3.4,2.7)rectangle(3.6,2.9) (3.6,2.7)rectangle(3.8,2.9) (3.8,2.7)rectangle(4,2.9) 
      (3.4,2.9)rectangle(3.6,3.1) 
      (3.4,3.1)rectangle(3.6,3.3) 
      
      (3.5,1.5)rectangle(3.7,1.7) (3.7,1.5)rectangle(3.9,1.7)
      (3.5,1.7)rectangle(3.7,1.9)
      (3.5,1.9)rectangle(3.7,2.1)
      
      (3.6,.2)rectangle(3.8,.4)
      (3.6,.4)rectangle(3.8,.6)
      (3.6,.6)rectangle(3.8,.8)
      (4.5,5.2)rectangle(4.7,5.4) (4.7,5.2)rectangle(4.9,5.4) (4.9,5.2)rectangle(5.1,5.4) (5.1,5.2)rectangle(5.3,5.4)                 (5.3,5.2)rectangle(5.5,5.4) (5.5,5.2)rectangle(5.7,5.4)
      (4.5,5.4)rectangle(4.7,5.6) (4.7,5.4)rectangle(4.9,5.6) (4.9,5.4)rectangle(5.1,5.6) (5.1,5.4)rectangle(5.3,5.6)
      (4.5,5.6)rectangle(4.7,5.8) (4.7,5.6)rectangle(4.9,5.8)
      (6,6.4)rectangle(6.2,6.6) (6.2,6.4)rectangle(6.4,6.6) (6.4,6.4)rectangle(6.6,6.6) (6.6,6.4)rectangle(6.8,6.6)
      (6.8,6.4)rectangle(7,6.6) (7,6.4)rectangle(7.2,6.6) (7.2,6.4)rectangle(7.4,6.6)
      (6,6.6)rectangle(6.2,6.8) (6.2,6.6)rectangle(6.4,6.8) (6.4,6.6)rectangle(6.6,6.8) (6.6,6.6)rectangle(6.8,6.8)
      (6.8,6.6)rectangle(7,6.8)
      (6,6.8)rectangle(6.2,7) (6.2,6.8)rectangle(6.4,7) (6.4,6.8)rectangle(6.6,7)
      (6,7)rectangle(6.2,7.2)
      
      (6.2,5.2)rectangle(6.4,5.4) (6.4,5.2)rectangle(6.6,5.4) (6.6,5.2)rectangle(6.8,5.4)
      (6.8,5.2)rectangle(7,5.4) (7,5.2)rectangle(7.2,5.4)
      (6.2,5.4)rectangle(6.4,5.6) (6.4,5.4)rectangle(6.6,5.6) (6.6,5.4)rectangle(6.8,5.6)
      (6.2,5.6)rectangle(6.4,5.8)
      (6.2,5.8)rectangle(6.4,6)
      
      (6.3,4)rectangle(6.5,4.2) (6.5,4)rectangle(6.7,4.2) (6.7,4)rectangle(6.9,4.2) (6.9,4)rectangle(7.1,4.2)
      (6.3,4.2)rectangle(6.5,4.4) (6.5,4.2)rectangle(6.7,4.4) 
      (6.3,4.4)rectangle(6.5,4.6) 
      (6.3,4.6)rectangle(6.5,4.8)
        
      (6.4,2.8)rectangle(6.6,3) (6.6,2.8)rectangle(6.8,3) (6.8,2.8)rectangle(7,3) 
      (6.4,3)rectangle(6.6,3.2) 
      (6.4,3.2)rectangle(6.6,3.4) 
      (6.4,3.4)rectangle(6.6,3.6)
           
      (6.5,1.6)rectangle(6.7,1.8) (6.7,1.6)rectangle(6.9,1.8)
      (6.5,1.8)rectangle(6.7,2) 
      (6.5,2)rectangle(6.7,2.2)
      (6.5,2.2)rectangle(6.7,2.4)
           
      (6.6,.1)rectangle(6.8,.3)
      (6.6,.3)rectangle(6.8,.5)
      (6.6,.5)rectangle(6.8,.7)
      (6.6,.7)rectangle(6.8,.9)
      
      (7.7,1.6)rectangle(7.9,1.8) (7.9,1.6)rectangle(8.1,1.8)
      (7.7,1.8)rectangle(7.9,2) (7.9,1.8)rectangle(8.1,2)
      (7.7,2)rectangle(7.9,2.2)
      (7.7,2.2)rectangle(7.9,2.4);

\draw [dotted] (.35,0)--(.35,1.5) (1.05,-.3)--(1.05,2.5) (1.95,-.6)--(1.95,3.5) (3.05,-.8)--(3.05,4.8) (4.35,0)--(4.35,5.9)                          (5.85,0)--(5.85,7.3);
\draw (-.2,-.1)node{\tiny{degree 0}} (.2,-.4)node{\tiny{degree 1}} (1,-.6)node{\tiny{degree 2}} (1.8,-.8)node{\tiny{degree 3}}; 
\draw (3.5,-.8)node{...};
\draw (5,-1) node{degree 6};
\draw [->] (.75,-.38)--(1.05,-.38);
\draw [->] (1.65,-.58)--(1.95,-.58);
\draw [->] (2.45,-.78)--(3.05,-.78);
\draw [->] (6,-1)--(8,-1);
\draw [<-] (-.3,-1)--(4,-1);

\draw [->] (.25,.5)--(1.2,.5);
\draw [->] (1.8,.5)--(3.4,.5);
\draw [->] (4.1,.5)--(6.4,.5);
\draw [->] (1.1,1.3)--(3.25,1.8);
\draw [white, line width=3pt] (1.5,.9)--(1.5,1.9);
\draw [->] (1.5,.9)--(1.5,1.9);
\draw [->] (2,2.2)--(3.1,2.7);
\draw [->] (4.1,1.9)--(6.1,1.9);
\draw [->] (4.15,3.05)--(5.9,3.05);
\draw [->] (4.3,4.4)--(5.9,5);
\draw [->] (3.7,.9)--(3.7,1.4);
\draw [->] (3.7,2.2)--(3.7,2.6);
\draw [->] (3.7,3.5)--(3.7,3.9);
\draw [->] (6.7,1)--(6.7,1.5);
\draw [->] (6.7,2.5)--(6.7,2.7);
\draw [->] (6.7,3.3)--(6.7,3.9);
\draw [->] (6.7,4.5)--(6.7,5.1);
\draw [->] (6.7,5.7)--(6.7,6.35);
\draw [->] (7,2)--(7.5,2);
\draw [->] (7.75,2.5)--(7,3.9);

\draw [->] (.1,.7)--(.5,1.1);
\draw [->] (.9,1.5)--(1.3,1.9);
\draw [->] (1.8,2.4)--(2.2,2.8);
\draw [->] (2.8,3.4)--(3.2,3.8);
\draw [->] (3.9,4.5)--(4.4,5);
\draw [->] (5.1,5.7)--(5.7,6.3);

\draw [white, line width=3pt] (.6,.1) ..controls (3,.1) and (3.2,2)..(3.2,3);
\draw [white, line width=3pt] (3.2,3) ..controls (3.2,3.9) and (2.2,3.9)..(2,3.9);
\draw [white, line width=3pt] (2,3.9) ..controls (.8,3.9) and (-.6,2)..(-.6,.7);
\draw [white, line width=3pt] (-.6,.7) ..controls (-.6,.2) and (-.2,.1)..(.6,.1);

\draw [white, line width=3pt] (7,1.3) ..controls (8.5,1.3) and (8.5,1.9)..(8.5,2.3);
\draw [white, line width=3pt] (8.5,2.3) ..controls (8.5,4) and (7.8,6.15)..(6.7,6.15);
\draw [white, line width=3pt] (6.7,6.15) ..controls (5.8,6.15) and (6,5.9)..(6,3.5);
\draw [white, line width=3pt] (6,3.5) ..controls (6,1.8) and (6.2,1.3)..(7,1.3);
\draw [gr, very thin] (.6,.1) ..controls (3,.1) and (3.2,2)..(3.2,3);
\draw [gr, very thin] (3.2,3) ..controls (3.2,3.9) and (2.2,3.9)..(2,3.9);
\draw [gr, very thin] (2,3.9) ..controls (.8,3.9) and (-.6,2)..(-.6,.7);
\draw [gr, very thin] (-.6,.7) ..controls (-.6,.2) and (-.2,.1)..(.6,.1);

\draw [gr, very thin] (7,1.3) ..controls (8.5,1.3) and (8.5,1.9)..(8.5,2.3);
\draw [gr, very thin] (8.5,2.3) ..controls (8.5,4) and (7.8,6.15)..(6.7,6.15);
\draw [gr, very thin] (6.7,6.15) ..controls (5.8,6.15) and (6,5.9)..(6,3.5);
\draw [gr, very thin] (6,3.5) ..controls (6,1.8) and (6.2,1.3)..(7,1.3);
\end{tikzpicture}
}
\end{picture}

The encircled parts show the recursive structure of the extremal points by using the $*$-operator. 
\end{exa}

There is a natural partial ordering in $\mathbb{Q}^{d+1}$:
$$h \leq g \quad :\Leftrightarrow \quad h_i \leq g_i \;\text{for all}\; i=0,\ldots,d.$$
The chains in this partial ordering offer totally ordered subsets in $\mathbb{Q}^{d+1}$.\\
In fact, this partial ordering harmonizes perfectly with the visualization of the h-vectors of $R$-algebras by staircases of lex-segment ideals by embedding the box-diagrams in each other, illustrated in the previous example by the arrows. Moreover, from any h-diagram we get a stack of staircases embedded consecutively and that involves a chain of h-vectors of $R$-algebras.

The decomposition algorithm that we use to prove theorem (\ref{extr}) leads to an h-diagram with staircases of extremal points in each level and as already mentioned this implies a totally ordered chain in the usual partial ordering in \Hc$(d)$.

\begin{cor}\label{cor}
Every element $h \in$ \Hc$(d)$ \;can be written as $h =\sum_{i \in I} q_i\! \cdot \!v^i,$\, with \linebreak$  v^i \in \Ex(d)$, where the $(v^i)_{i\in I}$ form a totally ordered chain.
\end{cor}

This decomposition does not have to be unique even with a total order:

\begin{exa}
 Let $n=2, \\ h=(2,1,2,0,1)=(1,1,1,0,1)+(1,0,1,0,0)=t^4*s^0 \, +\,t^2 = $\\
 \hspace*{28.7mm} $=(1,1,2,0,1)+(1,0,0,0,0)=t^4*s^1\!+\!s^0.$\\
 Both decompositions are totally ordered.
\end{exa}

\section{Proof of Theorem 2.3.}

To prove the theorem we use an algorithm that decomposes any h-vector of an artinian graded module over $R$ finitely generated in degree 0.

To get the right intuition we illustrate an example.
\begin{exa}
Let $\deg(y)=3$ \,and\, $h=(3,3,2,4,2,1,2,1)=\\ =(1,1,1,2,1,1,1,0)\;+\;\;(1,1,0,1,1,0,1,1)\;\;+\;\;(1,1,1,1,0,0,0,0)$\\
\col{Giving every degree its own color the corresponding staircases look like this:}
\bw{The corresponding staircases look like this:}

\col{
\unitlength1cm
\begin{picture}(11,2.25)
\put(1.2,.4){
\begin{tikzpicture}
\filldraw [fill=plum,draw=black](0,0)rectangle(.5,.5);
\filldraw [fill=blueberry,draw=black](.5,0)rectangle(1,.5);
\filldraw [fill=eggplant,draw=black](1,0)rectangle(1.5,.5);
\filldraw [fill=cranberry,draw=black](1.5,0)rectangle(2,.5);
\filldraw [fill=cranberry,draw=black](0,.5)rectangle(.5,1);
\filldraw [fill=beans,draw=black] (.5,.5)rectangle(1,1);
\filldraw [fill=apricot,draw=black](1,.5)rectangle(1.5,1);
\filldraw [fill=corn,draw=black] (0,1)rectangle(.5,1.5);     
\end{tikzpicture}
}

\put(5.5,.4){
\begin{tikzpicture}
\filldraw [fill=plum,draw=black](0,0)rectangle(.5,.5);
\filldraw [fill=blueberry,draw=black](.5,0)rectangle(1,.5);
\filldraw [fill=cranberry,draw=black](0,.5)rectangle(.5,1);
\filldraw [fill=beans,draw=black] (.5,.5)rectangle(1,1);
\filldraw [fill=corn,draw=black](0,1)rectangle(.5,1.5);
\filldraw [fill=llemon,draw=black](.5,1)rectangle(1,1.5);      
\end{tikzpicture}
}

\put(9.5,.4){
\begin{tikzpicture}
\filldraw [fill=plum,draw=black](0,0)rectangle(.5,.5);
\filldraw [fill=blueberry,draw=black](.5,0)rectangle(1,.5);
\filldraw [fill=eggplant,draw=black] (1,0)rectangle(1.5,.5); 
\filldraw [fill=cranberry,draw=black] (0,.5)rectangle(.5,1);    
\end{tikzpicture}
}
\end{picture}
}
\bw{
\unitlength1cm
\begin{picture}(11,2.25)
\put(1.2,.4){
\begin{tikzpicture}
\draw (0,0)rectangle(.5,.5) (.5,0)rectangle(1,.5)  (1,0)rectangle(1.5,.5) (1.5,0)rectangle(2,.5);
\draw (0,.5)rectangle(.5,1) (.5,.5)rectangle(1,1) (1,.5)rectangle(1.5,1);
\draw (0,1)rectangle(.5,1.5);  
\draw (.25,.25) node{0} (.75,.25)node{1} (1.25,.25)node{2} (1.75,.25)node{3};
\draw (.25,.75) node{3} (.75,.75)node{4} (1.25,.75)node{5};
\draw (.25,1.25) node{6};  
\end{tikzpicture}
}

\put(5.5,.4){
\begin{tikzpicture}
\draw (0,0)rectangle(.5,.5) (.5,0)rectangle(1,.5);
\draw (0,.5)rectangle(.5,1) (.5,.5)rectangle(1,1);
\draw (0,1)rectangle(.5,1.5) (.5,1)rectangle(1,1.5); 
\draw (.25,.25) node{0} (.75,.25)node{1} ;
\draw (.25,.75) node{3} (.75,.75)node{4} ;
\draw (.25,1.25) node{6} (.75,1.25)node{7};       
\end{tikzpicture}
}

\put(9.5,.4){
\begin{tikzpicture}
\draw (0,0)rectangle(.5,.5) (.5,0)rectangle(1,.5) (1,0)rectangle(1.5,.5); 
\draw (0,.5)rectangle(.5,1); 
\draw (.25,.25) node{0} (.75,.25)node{1} (1.25,.25)node{2};
\draw (.25,.75) node{3};   
\end{tikzpicture}
}
\end{picture}
}

In this case we get as an $h$-diagram:

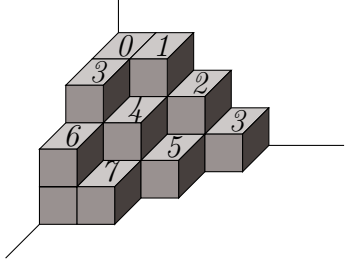
\unitlength1cm
\bw{
\begin{picture}(8,3.5)
\put(2,0){
\begin{tikzpicture}
 \draw(0,0)--(3,0);
 \draw(0,0)--(0,2);
 \draw(0,0)--(-1.5,-1.5);

 \filldraw  [fill=lgr,draw=black](0,1.5)--(-.5*.7,1.5-.5*.7)--(.5-.5*.7,-.5*.7+1.5)--(.5,1.5)--(0,1.5);
\draw (.45-.5*.7,1.68-.5*.7) node{\emph{0}}; 


 \filldraw  [fill=lgr,draw=black] (-.5*.7,1.5-.5*.7)--(.5-.5*.7,-.5*.7+1.5)--(.5-1*.7,-1*.7+1.5)--(-1*.7,-1*.7+1.5)--(-.5*.7,1.5-.5*.7);
 \filldraw [fill=dgr,draw=black] (.5-.5*.7,-.5*.7+1.5)--(.5-1*.7,-1*.7+1.5)--(.5-1*.7,-1*.7+1)--(.5-.5*.7,-.5*.7+1)--(.5-.5*.7,-.5*.7+1.5);
 \filldraw [fill=gr,draw=black]  (-1*.7,-1*.7+1.5)rectangle(.5-1*.7,-1*.7+1);
\draw (.45-1*.7,1.68-1*.7) node{\emph{3}};

 \filldraw [fill=lgr,draw=black] (-1*.7,-1*.7+1)--(.5-1*.7,-1*.7+1)--(.5-1.5*.7,-1.5*.7+1)--(-1.5*.7,-1.5*.7+1)--(-1*.7,-1*.7+1);
 \filldraw  [fill=dgr,draw=black] (.5-1*.7,-1*.7+1)--(.5-1.5*.7,-1.5*.7+1)--(.5-1.5*.7,-1.5*.7+.5)--(.5-1*.7,-1*.7+.5)-- (.5-1*.7,-1*.7+1);   
 \filldraw [fill=gr,draw=black]   (-1.5*.7,1-1.5*.7)rectangle (.5-1.5*.7,.5-1.5*.7)
           (.5-1.5*.7,.5-1.5*.7)rectangle ( -1.5*.7,-1.5*.7) ;
\draw  (.45-1.5*.7,1.18-1.5*.7) node{\emph{6}};

 \filldraw [fill=lgr,draw=black] (.5-1.5*.7,.5-1.5*.7)-- (.5-1*.7,.5-1*.7)-- (1-1*.7,.5-1*.7)-- (1-1.5*.7,.5-1.5*.7)-- (.5-1.5*.7,.5-1.5*.7);
 \filldraw  [fill=dgr,draw=black]  (1-1*.7,.5-1*.7)-- (1-1.5*.7,.5-1.5*.7)--(1-1.5*.7,-1.5*.7)--(1-1*.7,-1*.7)--(1-1*.7,.5-1*.7);
 \filldraw [fill=gr,draw=black]  (.5-1.5*.7,-1.5*.7)rectangle(1-1.5*.7,.5-1.5*.7);
\draw  (.95-1.5*.7,.67-1.5*.7) node{\emph{7}};

\filldraw  [fill=lgr,draw=black](.5-1*.7,1-1*.7)-- (.5-.5*.7,1-.5*.7)-- (1-.5*.7,1-.5*.7)-- (1-1*.7,1-1*.7)-- (.5-1*.7,1-1*.7);
\filldraw [fill=dgr,draw=black]  (1-.5*.7,1-.5*.7)--(1-1*.7,1-1*.7)--(1-1*.7,.5-1*.7)--(1-.5*.7,.5-.5*.7)--(1-.5*.7,1-.5*.7);
\filldraw [fill=gr,draw=black]  (1-1*.7,1-1*.7)rectangle(.5-1*.7,.5-1*.7);
\draw (.95-1*.7,1.18-1*.7) node{\emph{4}};

\filldraw  [fill=lgr,draw=black](1-1*.7,.5-1*.7)-- (1-.5*.7,.5-.5*.7)-- (1.5-.5*.7,.5-.5*.7)-- (1.5-1*.7,.5-1*.7)-- (1-1*.7,.5-1*.7);
\filldraw [fill=dgr,draw=black]  (1.5-.5*.7,.5-.5*.7)--(1.5-1*.7,.5-1*.7)--(1.5-1*.7,-1*.7)--(1.5-.5*.7,-.5*.7)--(1.5-.5*.7,.5-.5*.7);
\filldraw [fill=gr,draw=black]  (1.5-1*.7,.5-1*.7)rectangle(1-1*.7,-1*.7);
\draw (1.45-1*.7,.68-1*.7) node{\emph{5}};
 
 \filldraw [fill=lgr,draw=black] (.5,1.5)--(.5-.5*.7,1.5-.5*.7)--(1-.5*.7,1.5-.5*.7)--(1,1.5)--(.5,1.5);
\filldraw  [fill=dgr,draw=black] (1,1.5)--(1-.5*.7,1.5-.5*.7)--(1-.5*.7,1-.5*.7)--(1,1)--(1,1.5);
\filldraw [fill=gr,draw=black] (.5-.5*.7,1.5-.5*.7)rectangle(1-.5*.7,1-.5*.7);
\draw (.95-.5*.7,1.68-.5*.7) node{\emph{1}};

 \filldraw [fill=lgr,draw=black]  (1,1)--(1.5,1)--(1.5-.5*.7,1-.5*.7)--(1-.5*.7,1-.5*.7)--(1,1);
 \filldraw  [fill=dgr,draw=black] (1.5,1)--(1.5,.5)--(1.5-.5*.7,.5-.5*.7)--(1.5-.5*.7,1-.5*.7)--(1.5,1);
 \filldraw [fill=gr,draw=black] (1-.5*.7,.5-.5*.7)rectangle(1.5-.5*.7,1-.5*.7);
\draw (1.45-.5*.7,1.18-.5*.7) node{\emph{2}};

 \filldraw [fill=lgr,draw=black]  (1.5,.5)--(2,.5)--(2-.5*.7,.5-.5*.7)--(1.5-.5*.7,.5-.5*.7)--(1.5,.5);
 \filldraw  [fill=dgr,draw=black] (2,.5)--(2,0)--(2-.5*.7,-.5*.7)--(2-.5*.7,.5-.5*.7)--(2,.5);
 \filldraw  [fill=gr,draw=black] (1.5-.5*.7,.5-.5*.7)rectangle(2-.5*.7,-.5*.7);
\draw (1.95-.5*.7,.68-.5*.7) node{\emph{3}};
 
\end{tikzpicture}
}
\end{picture}
}

We can rearrange the boxes in a way so that all the \col{colors respectively }degrees only sit over the allowed spots dedicated by staircases. In every level the area should be filled out as large as possible and take the maximal height in this level, however the rows and columns should still decrease. In this process the boxes can be cut horizontally in every rational proportion, so the levels can have any positive rational height. 
In our example we get the following stack:

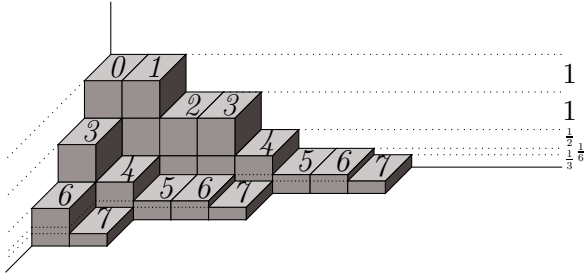
\bw{
\begin{picture}(10,3.8)
\put(2,0){
\begin{tikzpicture}
 \draw(0,0)--(6,0);
 \draw(0,0)--(0,2.2);
 \draw(0,0)--(-2*.7,-2*.7);

 
\filldraw  [fill=lgr,draw=black](0,1.5)--(-.5*.7,1.5-.5*.7)--(.5-.5*.7,-.5*.7+1.5)--(.5,1.5)--(0,1.5);
\filldraw  [fill=gr,draw=black] (-.5*.7,1.5-.5*.7)rectangle(.5-.5*.7,1-.5*.7);
\draw (.45-.5*.7,1.68-.5*.7) node{\emph{0}}; 
 
\filldraw [fill=lgr,draw=black] (.5,1.5)--(.5-.5*.7,1.5-.5*.7)--(1-.5*.7,1.5-.5*.7)--(1,1.5)--(.5,1.5);
\filldraw [fill=dgr,draw=black] (1,1.5)--(1-.5*.7,1.5-.5*.7)--(1-.5*.7,1-.5*.7)--(1,1)--(1,1.5);
\filldraw [fill=gr,draw=black] (.5-.5*.7,1.5-.5*.7)rectangle(1-.5*.7,1-.5*.7);
\filldraw [fill=gr,draw=black] (.5-.5*.7,1-.5*.7)rectangle(1-.5*.7,.5-.5*.7);
\draw (.95-.5*.7,1.68-.5*.7) node{\emph{1}};

\filldraw [fill=lgr,draw=black]  (1,1)--(1.5,1)--(1.5-.5*.7,1-.5*.7)--(1-.5*.7,1-.5*.7)--(1,1);
\filldraw [fill=dgr,draw=black] (1.5,1)--(1.5,.5)--(1.5-.5*.7,.5-.5*.7)--(1.5-.5*.7,1-.5*.7)--(1.5,1);
\filldraw [fill=gr,draw=black] (1-.5*.7,.5-.5*.7)rectangle(1.5-.5*.7,1-.5*.7);
\filldraw [fill=gr,draw=black] (1-.5*.7,-.5*.7)rectangle(1.5-.5*.7,.5-.5*.7);
\draw (1.45-.5*.7,1.18-.5*.7) node{\emph{2}};

\filldraw [fill=lgr,draw=black]  (1.5,1)--(2,1)--(2-.5*.7,1-.5*.7)--(1.5-.5*.7,1-.5*.7)--(1.5,1);
\filldraw [fill=dgr,draw=black] (2,1)--(2,.5)--(2-.5*.7,.5-.5*.7)--(2-.5*.7,1-.5*.7)--(2,1);
\filldraw [fill=dgr,draw=black] (2,.5)--(2,0)--(2-.5*.7,-.5*.7)--(2-.5*.7,.5-.5*.7)--(2,.5);
\filldraw [fill=gr,draw=black] (1.5-.5*.7,1-.5*.7)rectangle(2-.5*.7,.5-.5*.7);
\filldraw [fill=gr,draw=black] (1.5-.5*.7,.5-.5*.7)rectangle(2-.5*.7,-.5*.7);
\draw (1.95-.5*.7,1.18-.5*.7) node{\emph{3}};

\filldraw [fill=lgr,draw=black]  (2,.5)--(2.5,.5)--(2.5-.5*.7,.5-.5*.7)--(2-.5*.7,.5-.5*.7)--(2,.5);
\filldraw [fill=dgr,draw=black] (2.5,.5)--(2.5,0)--(2.5-.5*.7,-.5*.7)--(2.5-.5*.7,.5-.5*.7)--(2.5,.5);
\filldraw [fill=gr,draw=black] (2-.5*.7,.5-.5*.7)rectangle(2.5-.5*.7,-.5*.7);
\draw (2.45-.5*.7,.68-.5*.7) node{\emph{4}};

\filldraw [fill=lgr,draw=black]  (2.5,.25)--(3,.25)--(3-.5*.7,.25-.5*.7)--(2.5-.5*.7,.25-.5*.7)--(2.5,.25);
\filldraw  [fill=gr,draw=black] (2.5-.5*.7,.25-.5*.7)rectangle(3-.5*.7,-.5*.7);
\draw (2.95-.5*.7,.43-.5*.7) node{\emph{5}};

\filldraw [fill=lgr,draw=black]  (3,.25)--(3.5,.25)--(3.5-.5*.7,.25-.5*.7)--(3-.5*.7,.25-.5*.7)--(3,.25);
\filldraw  [fill=dgr,draw=black] (3.5,.25)--(3.5,0)--(3.5-.5*.7,-.5*.7)--(3.5-.5*.7,.25-.5*.7)--(3.5,.25);
\filldraw  [fill=gr,draw=black] (3-.5*.7,.25-.5*.7)rectangle(3.5-.5*.7,-.5*.7);
\draw (3.45-.5*.7,.43-.5*.7) node{\emph{6}};

\filldraw [fill=lgr,draw=black]  (3.5,.33*.5)--(4,.33*.5)--(4-.5*.7,.33*.5-.5*.7)--(3.5-.5*.7,.33*.5-.5*.7)--(3.5,.33*.5);
\filldraw  [fill=dgr,draw=black] (4,.33*.5)--(4,0)--(4-.5*.7,-.5*.7)--(4-.5*.7,.33*.5-.5*.7)--(4,.33*.5);
\filldraw  [fill=gr,draw=black] (3.5-.5*.7,.33*.5-.5*.7)rectangle(4-.5*.7,-.5*.7);
\draw (3.95-.5*.7,.33-.5*.7) node{\emph{7}};

\draw [densely dotted](3.5-.5*.7,.33*.5-.5*.7)--(2-.5*.7,.33*.5-.5*.7) (2.5-.5*.7,.25-.5*.7)--(2-.5*.7,.25-.5*.7);

\filldraw  [fill=lgr,draw=black] (-.5*.7,1-.5*.7)--(.5-.5*.7,1-.5*.7)--(.5-1*.7,1-1*.7)--(-1*.7,1-1*.7)--(-.5*.7,1-.5*.7);
\filldraw [fill=dgr,draw=black] (.5-.5*.7,1-.5*.7)--(.5-1*.7,1-1*.7)--(.5-1*.7,.5-1*.7)--(.5-.5*.7,.5-.5*.7)--(.5-.5*.7,1-.5*.7);
\filldraw  [fill=gr,draw=black] (-1*.7,1-1*.7)rectangle(.5-1*.7,.5-1*.7);
\draw (.45-1*.7,1.18-1*.7) node{\emph{3}};

\filldraw  [fill=lgr,draw=black](.5-1*.7,.5-1*.7)-- (.5-.5*.7,.5-.5*.7)-- (1-.5*.7,.5-.5*.7)-- (1-1*.7,.5-1*.7)-- (.5-1*.7,.5-1*.7);
\filldraw [fill=dgr,draw=black]  (1-.5*.7,.5-.5*.7)--(1-1*.7,.5-1*.7)--(1-1*.7,-1*.7)--(1-.5*.7,-.5*.7)--(1-.5*.7,.5-.5*.7);
\filldraw [fill=gr,draw=black]  (1-1*.7,.5-1*.7)rectangle(.5-1*.7,-1*.7);
\draw (.95-1*.7,.68-1*.7) node{\emph{4}};

\filldraw  [fill=lgr,draw=black](1-1*.7,.25-1*.7)-- (1-.5*.7,.25-.5*.7)-- (1.5-.5*.7,.25-.5*.7)-- (1.5-1*.7,.25-1*.7)-- (1-1*.7,.25-1*.7);
\filldraw [fill=dgr,draw=black]  (1.5-.5*.7,.25-.5*.7)--(1.5-1*.7,.25-1*.7)--(1.5-1*.7,-1*.7)--(1.5-.5*.7,-.5*.7)--(1.5-.5*.7,.25-.5*.7);
\filldraw [fill=gr,draw=black]  (1.5-1*.7,.25-1*.7)rectangle(1-1*.7,-1*.7);
\draw (1.45-1*.7,.43-1*.7) node{\emph{5}};

\filldraw  [fill=lgr,draw=black](1.5-1*.7,.25-1*.7)-- (1.5-.5*.7,.25-.5*.7)-- (2-.5*.7,.25-.5*.7)-- (2-1*.7,.25-1*.7)-- (1.5-1*.7,.25-1*.7);
\filldraw [fill=dgr,draw=black]  (2-.5*.7,.25-.5*.7)--(2-1*.7,.25-1*.7)--(2-1*.7,-1*.7)--(2-.5*.7,-.5*.7)--(2-.5*.7,.25-.5*.7);
\filldraw [fill=gr,draw=black]  (2-1*.7,.25-1*.7)rectangle(1.5-1*.7,-1*.7);
\draw (1.95-1*.7,.43-1*.7) node{\emph{6}};

\filldraw  [fill=lgr,draw=black](2-1*.7,.66*.25-1*.7)-- (2-.5*.7,.66*.25-.5*.7)-- (2.5-.5*.7,.66*.25-.5*.7)-- (2.5-1*.7,.66*.25-1*.7)-- (2-1*.7,.66*.25-1*.7);
\filldraw [fill=dgr,draw=black]  (2.5-.5*.7,.66*.25-.5*.7)--(2.5-1*.7,.66*.25-1*.7)--(2.5-1*.7,-1*.7)--(2.5-.5*.7,-.5*.7)--(2.5-.5*.7,.66*.25-.5*.7);
\filldraw [fill=gr,draw=black]  (2.5-1*.7,.66*.25-1*.7)rectangle(2-1*.7,-1*.7);
\draw (2.45-1*.7,.33-1*.7) node{\emph{7}};

\draw [densely dotted] (2-1*.7,.33*.5-1*.7)--(1-1*.7,.33*.5-1*.7) (1-1*.7,.25-1*.7)--(.5-1*.7,.25-1*.7) ;

\filldraw [fill=lgr,draw=black] (-1*.7,.5-1*.7)--(.5-1*.7,.5-1*.7)--(.5-1.5*.7,.5-1.5*.7)--(-1.5*.7,.5-1.5*.7)--(-1*.7,.5-1*.7);
\filldraw  [fill=dgr,draw=black] (.5-1*.7,.5-1*.7)--(.5-1.5*.7,.5-1.5*.7)--(.5-1.5*.7,-1.5*.7)--(.5-1*.7,-1*.7)-- (.5-1*.7,.5-1*.7);   
\filldraw [fill=gr,draw=black] (.5-1.5*.7,.5-1.5*.7)rectangle ( -1.5*.7,-1.5*.7) ;
\draw (.45-1.5*.7,.68-1.5*.7) node{\emph{6}};

\filldraw [fill=lgr,draw=black] (.5-1.5*.7,.33*.5-1.5*.7)-- (.5-1*.7,.33*.5-1*.7)-- (1-1*.7,.33*.5-1*.7)-- (1-1.5*.7,.33*.5-1.5*.7)-- (.5-1.5*.7,.33*.5-1.5*.7);
\filldraw  [fill=dgr,draw=black]  (1-1*.7,.33*.5-1*.7)-- (1-1.5*.7,.33*.5-1.5*.7)--(1-1.5*.7,-1.5*.7)--(1-1*.7,-1*.7)--(1-1*.7,.33*.5-1*.7);
\filldraw [fill=gr,draw=black]  (.5-1.5*.7,-1.5*.7)rectangle(1-1.5*.7,.33*.5-1.5*.7);
\draw (.95-1.5*.7,.33-1.5*.7) node{\emph{7}};

\draw [densely dotted] (-1.5*.7,.33*.5-1.5*.7)--(.5-1.5*.7,.33*.5-1.5*.7) (-1.5*.7,.25-1.5*.7)--(.5-1.5*.7,.25-1.5*.7) (.5-1*.7,.25-1*.7)--(.5-1.5*.7,.25-1.5*.7);

\draw [dotted] (4,.33*.5)--(6,.33*.5) (3.5,.25)--(6,.25) (2.5,.5)--(6,.5) (1.5,1)--(6,1) (1,1.5)--(6,1.5);
\draw [dotted] (-1.5*.7,.33*.5-1.5*.7)--(-2*.7,.33*.5-2*.7) (-1.5*.7,.25-1.5*.7)--(-2*.7,.25-2*.7) (-1.5*.7,.5-1.5*.7)--(-2*.7,.5-2*.7) (-1*.7,1-1*.7)--(-2*.7,1-2*.7) (-.5*.7,1.5-.5*.7)--(-2*.7,1.5-2*.7);
\draw (6.1,0.1) node{\scalebox{.5}{$\frac{1}{3}$}} (6.25,0.2) node{\scalebox{.5}{$\frac{1}{6}$}} (6.1,0.4) node{\scalebox{.5}{$\frac{1}{2}$}} (6.1,.75) node{\small{1}} (6.1,1.25) node{\small{1}};
\end{tikzpicture}
}
\end{picture}
}

Running the algorithm coming up next we get the following decomposition:
$$h\;=\;\frac{1}{3} \cdot s^7 \;+\; \frac{1}{6} \cdot s^6 \;+ \;\frac{1}{2}(t^6 * s^3)\; + \;s^3 \;+\; s^1$$
\end{exa}

\pagebreak
The largest area that can be filled out up to a fixed degree is given by the staircase of the maximal h-vector of the same degree. If that does not fit, the maximal h-vector of lower degree must be combined with a tower on the left and so on.
 So we get layers of staircases of extremal points in a total order.
That is what the algorithm does. 
An explicit description  of the algorithm as a flowchart can be found in the appendix.

We now follow the algorithm and prove the accuracy of theorem (\ref{extr}) in several steps.
\\

 Starting with any element $h=(h_0,\dots,h_d)\in \mathbb{Q}_{\geq0}^{d+1}$ we want to decide whether it is in the cone \Hc(d) and if so we want to give a positive rational decomposition in extremal points.

First we subtract a rational multiple of the maximal h-vectors $s^d$ as large as possible, so the rest stays non-negative. We continue by lowering the degree in every loop until we end up in $h=(0)$ or we get a trivial entry in a coordinate with index smaller than the degree, but $h_d >0$.
Let's call this a \textit{reduced h-vector}.

\begin{lemma}\label{lem1}
If $g \in$ \Hc(d) then $h=g-q \cdot s^d$ with $q \leq \underset{i \in \{0,\cdots,d\}}{\min}\{\frac{g_i}{s_i}\}$ is still in \Hc.
\end{lemma}

\begin{proof}
We have a decomposition $g_i= \sum_{j=1}^{s_i} g_i^j$ with properties (1) and (2) of proposition (\ref{hcone}) and as $q\leq \min \{\frac{g_i}{s_i}\}$ we may assume that $g_i^j\geq q$  for all $i,j$.
We can achieve this for example by replacing the decomposition in the highest degree $d$ by $\bar{g}_d^{s_d} =g_d-q(s_d-1)$ and $\bar{g}_d^j=q$ for all $j<s_d$. To fill the gaps $g_i^j<q$ in the lower degrees we take the material from the other spots of same degree starting with the smallest row but never underrunning $q$ and always ensuring that the rows and columns are still decreasing looking upwards from degree $i$. 

Having such a decomposition we set $h_i^j=g_i^j-q$ for $h=(g_0-q, \ldots,g_d- q  s_d)$ . Then the computation $$\sum_{j=1}^{s_i}h_i^j=\sum_{j=1}^{s_i}g_i^j -q \cdot s_i = g_i - q \cdot s_i $$ shows that both conditions in proposition (\ref{hcone}) are fulfilled, hence $h\in\mathbb{H}$.
\end{proof}

\begin{exa}\label{notextr}
Let $h=t^{n\cdot m-1}$ be the tower of degree $n\cdot m-1$. The staircase corresponding to this h-vector is a rectangle of size $n \times (m+1)$. The previous lemma gives the decomposition in extremal points $\sum_{\ell=1}^m q_\ell\cdot s^{n\ell-1}$ with $$q_m=\frac{1}{m} \quad \text{and} \quad q_\ell=\frac{1}{\ell}-\sum_{k=\ell+1}^m \frac{1}{q_k}.$$

\end{exa}

\pagebreak
\begin{lemma}\label{lem3}
There is no reduced h-vector of degree $d=n\!\cdot\! m-1$.
\end{lemma}

\begin{proof}
Let $h$ be reduced, thus there exists $k < d$ with $ h_k=0$ but $h_d>0$. \linebreak Let $k=n\!\cdot\! m_k+r_k$ be maximal among those indices fufilling the previous conditions. From the decreasing of the columns in the h-diagram follows $ h_{n\ell+r_k}=0 $ \, for all $\ell \geq m_k $ and from the decreasing of the rows follows the contradiction \linebreak $0=h_{n(m-1)+r_k}\geq h_{nm-1}>0$. Therefore no such $k$ can exist.
\end{proof}


The central part of the proof of theorem (\ref{extr}) is given by the following lemma.

\begin{lemma}\label{lem2}
Every element of the cone \Hc(d) decomposes into a positive rational linear combination of extremal points.
\end{lemma}

\begin{proof}
The proof goes by induction over the degree $d$. \\Set $m= \lfloor\frac{d}{n} \rfloor$ and $r=d-mn$.

It is clear up to $d=n-1$, since there is only one row in the h-diagram and the extremal points are just of the form $s^\ell=(1,\dots,1),$ \, with \,$\ell \leq d$. So we can use lemma (\ref{lem1}) to decompose.

Let $h\in \mathbb{H}(d)$ with $h_d > 0$. By lemma (\ref{lem1}) and lemma (\ref{lem3}) we can assume that $h$ is reduced and $d\neq (n-1) \mathrm{mod}\; n$. For a given $h \in \mathbb{H}(d)$ there are many decompositions fulfilling proposition (\ref{hcone}). We choose one that is maximal in the truncated sum $\sum_{i,j} \bar{h}_i^j$ with $\bar{h}_i^j= h_d$ for $h_i^j \geq h_d$ and $\bar{h}_i^j =h_i^j$ else. We also can achieve this by shifting the boxes as far to the right as possible.

As $h$ is reduced there exists a $k<d$ with $h_k =0$. From descending row and column condition we get $h_{k+1}=0$ for all $k \neq (n-1) \mathrm{mod}\; n$ and $h_{k+n}=0$ for all $k+n<d$. Therefore $h_{mn-1}=0$ and $h_d^{s_d} =h_d$.

Coming back to the h-diagram we cut off anything to the right of column \linebreak $r+1$, saying next to the tower $t^d$ and call the diagram on the right side $h'$. A visualisation of this process and the next steps is given in example (\ref{cutoff}). The boxes of the new diagram are still nested in the corner with decreasing rows and columns therefore we get again an h-diagram,
that has by construction at most $m$ rows and $d-2r-2$ columns. Hence the maximal degree of $h'$ is $d'= d-2r-3$ and we can use induction to decompose $h'$, but we only do this up to height $h_d$.
Glueing this back to the first $r+1$ columns again we have an h-diagram. The layers of $h'$ up to height $h_d$ are corresponding to staircases of extremal points in $\Ex(d')$, which are glued by the $*$-operator to the towers $t^d$ with the corresponding height, or maybe there is nothing left to glue we just take the tower itself.

Next we cut off anything above height $h_d$ and get the upper h-diagram $h''$ with degree strictly less than $d$ and therefore we can use induction again. The lowest level of $h''$ is with respect to the partial ordering smaller than the lower adjacent level to height $h_d$ in the last completed h-diagram, because we shifted as many boxes as possible under height $h_d$ and decomposed $h'$ only up to height $h_d$. Therefore we get a decomposition in a total order as stated in corollary (\ref{cor}).
\end{proof}

\pagebreak
\begin{exa}\label{cutoff}
Let $n=4$\; and \,$h=(3,3,2,2,3,3,2,0,1,1)$, a possible h-diagram could look like this:

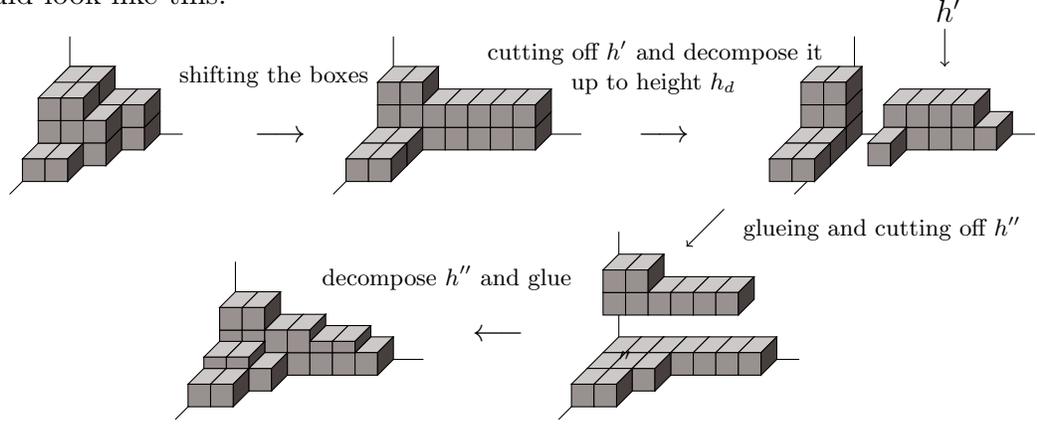
\unitlength1cm
\begin{picture}(13,6.3)
\put(0,3.8){
\put(.5,0){
\begin{tikzpicture}
\draw (0,0)--(1.5,0) (0,0)--(0,1.3)  (0,0)--(-.8,-.8);
\filldraw [fill=lgr,draw=black] (0,.3)--(-.3*.7,.3-.3*.7)--(.3-.3*.7,.3-.3*.7)--(.3,.3)--(0,.3);
\filldraw [fill=lgr,draw=black] (.3,.3)--(.3-.3*.7,.3-.3*.7)--(.6-.3*.7,.3-.3*.7)--(.6,.3)--(.3,.3);
\filldraw [fill=lgr,draw=black] (-.3*.7,.3-.3*.7)--(-.6*.7,.3-.6*.7)--(.3-.6*.7,.3-.6*.7)--(.3-.3*.7,.3-.3*.7)--(-.3*.7,.3-.3*.7);
\filldraw [fill=lgr,draw=black] (.3-.3*.7,.3-.3*.7)--(.3-.6*.7,.3-.6*.7)--(.6-.6*.7,.3-.6*.7)--(.6-.3*.7,.3-.3*.7)--(.3-.3*.7,.3-.3*.7);
\filldraw [fill=lgr,draw=black] (-.6*.7,.3-.6*.7)--(-.9*.7,.3-.9*.7)--(.3-.9*.7,.3-.9*.7)--(.3-.6*.7,.3-.6*.7)--(-.6*.7,.3-.6*.7);
\filldraw [fill=lgr,draw=black] (.3-.6*.7,.3-.6*.7)--(.3-.9*.7,.3-.9*.7)--(.6-.9*.7,.3-.9*.7)--(.6-.6*.7,.3-.6*.7)--(.3-.6*.7,.3-.6*.7);

\filldraw [fill=dgr,draw=black] (.6,.3)--(.6-.3*.7,.3-.3*.7)--(.6-.3*.7,-.3*.7)--(.6,0)--(.6,.3);
\filldraw [fill=dgr,draw=black] (.6-.3*.7,.3-.3*.7)--(.6-.6*.7,.3-.6*.7)--(.6-.6*.7,-.6*.7)--(.6-.3*.7,-.3*.7)--(.6-.3*.7,.3-.3*.7);
\filldraw [fill=dgr,draw=black] (.6-.6*.7,.3-.6*.7)--(.6-.9*.7,.3-.9*.7)--(.6-.9*.7,-.9*.7)--(.6-.6*.7,-.6*.7)--(.6-.6*.7,.3-.6*.7);

\filldraw [fill=gr,draw=black] (-.9*.7,.3-.9*.7)rectangle(.3-.9*.7,-.9*.7);
\filldraw [fill=gr,draw=black] (.3-.9*.7,.3-.9*.7)rectangle(.6-.9*.7,-.9*.7);

\filldraw [fill=lgr,draw=black] (.3,.9)--(.3-.6*.7,.9-.6*.7)--(.6-.6*.7,.9-.6*.7)--(.6,.9)--(.3,.9);
\filldraw [fill=lgr,draw=black] (0,.9)--(-.6*.7,.9-.6*.7)--(.3-.6*.7,.9-.6*.7)--(.3,.9)--(0,.9);
\filldraw [fill=lgr,draw=black] (0,.9)--(-.3*.7,.9-.3*.7)--(.3-.3*.7,.9-.3*.7)--(.3,.9)--(0,.9);
\filldraw [fill=lgr,draw=black] (.3,.9)--(.3-.3*.7,.9-.3*.7)--(.6-.3*.7,.9-.3*.7)--(.6,.9)--(.3,.9);
\filldraw [fill=dgr,draw=black] (.6,.9)--(.6-.3*.7,.9-.3*.7)--(.6-.3*.7,.6-.3*.7)--(.6,.6)--(.6,.9);
\filldraw [fill=dgr,draw=black] (.6-.3*.7,.9-.3*.7)--(.6-.6*.7,.9-.6*.7)--(.6-.6*.7,.6-.6*.7)--(.6-.3*.7,.6-.3*.7)--(.6-.3*.7,.9-.3*.7);
\filldraw [fill=gr,draw=black] (-.6*.7,.6-.6*.7)rectangle(.3-.6*.7,.3-.6*.7);
\filldraw [fill=gr,draw=black] (.3-.6*.7,.6-.6*.7)rectangle(.6-.6*.7,.3-.6*.7);
\filldraw [fill=gr,draw=black] (-.6*.7,.9-.6*.7)rectangle(.3-.6*.7,.6-.6*.7);
\filldraw [fill=gr,draw=black] (.3-.6*.7,.9-.6*.7)rectangle(.6-.6*.7,.6-.6*.7);

\filldraw [fill=lgr,draw=black] (.6,.6)--(.6-.3*.7,.6-.3*.7)--(.9-.3*.7,.6-.3*.7)--(.9,.6)--(.6,.6);
\filldraw [fill=lgr,draw=black] (.9,.6)--(.9-.3*.7,.6-.3*.7)--(1.2-.3*.7,.6-.3*.7)--(1.2,.6)--(.9,.6);
\filldraw [fill=lgr,draw=black] (.6-.3*.7,.6-.3*.7)--(.6-.6*.7,.6-.6*.7)--(.9-.6*.7,.6-.6*.7)--(.9-.3*.7,.6-.3*.7)--(.6-.3*.7,.6-.3*.7);
\filldraw [fill=dgr,draw=black] (1.2,.3)--(1.2-.3*.7,.3-.3*.7)--(1.2-.3*.7,-.3*.7)--(1.2,0)--(1.2,.3);
\filldraw [fill=dgr,draw=black] (1.2,.6)--(1.2-.3*.7,.6-.3*.7)--(1.2-.3*.7,.3-.3*.7)--(1.2,.3)--(1.2,.6);
\filldraw [fill=dgr,draw=black] (.9-.3*.7,.3-.3*.7)--(.9-.6*.7,.3-.6*.7)--(.9-.6*.7,-.6*.7)--(.9-.3*.7,-.3*.7)--(.9-.3*.7,.3-.3*.7);
\filldraw [fill=dgr,draw=black] (.9-.3*.7,.6-.3*.7)--(.9-.6*.7,.6-.6*.7)--(.9-.6*.7,.3-.6*.7)--(.9-.3*.7,.3-.3*.7)--(.9-.3*.7,.6-.3*.7);
\filldraw [fill=gr,draw=black] (.9-.3*.7,.3-.3*.7)rectangle(1.2-.3*.7,-.3*.7);
\filldraw [fill=gr,draw=black] (.9-.3*.7,.6-.3*.7)rectangle(1.2-.3*.7,.3-.3*.7);
\filldraw [fill=gr,draw=black] (.6-.6*.7,.6-.6*.7)rectangle(.9-.6*.7,.3-.6*.7);
\filldraw [fill=gr,draw=black] (.6-.6*.7,.3-.6*.7)rectangle(.9-.6*.7,-.6*.7);

\end{tikzpicture}
}
\put(3.9,.7){$\longrightarrow$}
\put(2.9,1.5){\scriptsize{shifting the boxes}}

\put(4.8,0){
\begin{tikzpicture}
\draw (0,0)--(2.5,0) (0,0)--(0,1.3)  (0,0)--(-.8,-.8);
\filldraw [fill=lgr,draw=black] (0,.3)--(-.3*.7,.3-.3*.7)--(.3-.3*.7,.3-.3*.7)--(.3,.3)--(0,.3);
\filldraw [fill=lgr,draw=black] (.3,.3)--(.3-.3*.7,.3-.3*.7)--(.6-.3*.7,.3-.3*.7)--(.6,.3)--(.3,.3);
\filldraw [fill=lgr,draw=black] (-.3*.7,.3-.3*.7)--(-.6*.7,.3-.6*.7)--(.3-.6*.7,.3-.6*.7)--(.3-.3*.7,.3-.3*.7)--(-.3*.7,.3-.3*.7);
\filldraw [fill=lgr,draw=black] (.3-.3*.7,.3-.3*.7)--(.3-.6*.7,.3-.6*.7)--(.6-.6*.7,.3-.6*.7)--(.6-.3*.7,.3-.3*.7)--(.3-.3*.7,.3-.3*.7);
\filldraw [fill=lgr,draw=black] (-.6*.7,.3-.6*.7)--(-.9*.7,.3-.9*.7)--(.3-.9*.7,.3-.9*.7)--(.3-.6*.7,.3-.6*.7)--(-.6*.7,.3-.6*.7);
\filldraw [fill=lgr,draw=black] (.3-.6*.7,.3-.6*.7)--(.3-.9*.7,.3-.9*.7)--(.6-.9*.7,.3-.9*.7)--(.6-.6*.7,.3-.6*.7)--(.3-.6*.7,.3-.6*.7);

\filldraw [fill=dgr,draw=black] (.6,.3)--(.6-.3*.7,.3-.3*.7)--(.6-.3*.7,-.3*.7)--(.6,0)--(.6,.3);
\filldraw [fill=dgr,draw=black] (.6-.3*.7,.3-.3*.7)--(.6-.6*.7,.3-.6*.7)--(.6-.6*.7,-.6*.7)--(.6-.3*.7,-.3*.7)--(.6-.3*.7,.3-.3*.7);
\filldraw [fill=dgr,draw=black] (.6-.6*.7,.3-.6*.7)--(.6-.9*.7,.3-.9*.7)--(.6-.9*.7,-.9*.7)--(.6-.6*.7,-.6*.7)--(.6-.6*.7,.3-.6*.7);

\filldraw [fill=gr,draw=black] (-.9*.7,.3-.9*.7)rectangle(.3-.9*.7,-.9*.7);
\filldraw [fill=gr,draw=black] (.3-.9*.7,.3-.9*.7)rectangle(.6-.9*.7,-.9*.7);

\filldraw [fill=lgr,draw=black] (.3,.9)--(.3-.3*.7,.9-.3*.7)--(.6-.3*.7,.9-.3*.7)--(.6,.9)--(.3,.9);
\filldraw [fill=lgr,draw=black] (0,.9)--(-.3*.7,.9-.3*.7)--(.3-.3*.7,.9-.3*.7)--(.3,.9)--(0,.9);
\filldraw [fill=gr,draw=black] (-.3*.7,.6-.3*.7)rectangle(.3-.3*.7,.3-.3*.7);
\filldraw [fill=gr,draw=black] (.3-.3*.7,.6-.3*.7)rectangle(.6-.3*.7,.3-.3*.7);
\filldraw [fill=gr,draw=black] (-.3*.7,.9-.3*.7)rectangle(.3-.3*.7,.6-.3*.7);
\filldraw [fill=gr,draw=black] (.3-.3*.7,.9-.3*.7)rectangle(.6-.3*.7,.6-.3*.7);
\filldraw [fill=dgr,draw=black] (.6,.9)--(.6-.3*.7,.9-.3*.7)--(.6-.3*.7,.6-.3*.7)--(.6,.6)--(.6,.9);

\filldraw [fill=lgr,draw=black] (.6,.6)--(.6-.3*.7,.6-.3*.7)--(.9-.3*.7,.6-.3*.7)--(.9,.6)--(.6,.6);
\filldraw [fill=lgr,draw=black] (.9,.6)--(.9-.3*.7,.6-.3*.7)--(1.2-.3*.7,.6-.3*.7)--(1.2,.6)--(.9,.6);
\filldraw [fill=lgr,draw=black] (1.2,.6)--(1.2-.3*.7,.6-.3*.7)--(1.5-.3*.7,.6-.3*.7)--(1.5,.6)--(1.2,.6);
\filldraw [fill=lgr,draw=black] (1.5,.6)--(1.5-.3*.7,.6-.3*.7)--(1.8-.3*.7,.6-.3*.7)--(1.8,.6)--(1.5,.6);
\filldraw [fill=lgr,draw=black] (1.8,.6)--(1.8-.3*.7,.6-.3*.7)--(2.1-.3*.7,.6-.3*.7)--(2.1,.6)--(1.8,.6);
\filldraw [fill=gr,draw=black] (.6-.3*.7,.3-.3*.7)rectangle(.9-.3*.7,-.3*.7);
\filldraw [fill=gr,draw=black] (.6-.3*.7,.6-.3*.7)rectangle(.9-.3*.7,.3-.3*.7);
\filldraw [fill=gr,draw=black] (.9-.3*.7,.3-.3*.7)rectangle(1.2-.3*.7,-.3*.7);
\filldraw [fill=gr,draw=black] (.9-.3*.7,.6-.3*.7)rectangle(1.2-.3*.7,.3-.3*.7);
\filldraw [fill=gr,draw=black] (1.2-.3*.7,.3-.3*.7)rectangle(1.5-.3*.7,-.3*.7);
\filldraw [fill=gr,draw=black] (1.2-.3*.7,.6-.3*.7)rectangle(1.5-.3*.7,.3-.3*.7);
\filldraw [fill=gr,draw=black] (1.5-.3*.7,.3-.3*.7)rectangle(1.8-.3*.7,-.3*.7);
\filldraw [fill=gr,draw=black] (1.5-.3*.7,.6-.3*.7)rectangle(1.8-.3*.7,.3-.3*.7);
\filldraw [fill=gr,draw=black] (1.8-.3*.7,.3-.3*.7)rectangle(2.1-.3*.7,-.3*.7);
\filldraw [fill=gr,draw=black] (1.8-.3*.7,.6-.3*.7)rectangle(2.1-.3*.7,.3-.3*.7);
\filldraw [fill=dgr,draw=black] (2.1,.3)--(2.1-.3*.7,.3-.3*.7)--(2.1-.3*.7,-.3*.7)--(2.1,0)--(2.1,.3);
\filldraw [fill=dgr,draw=black] (2.1,.6)--(2.1-.3*.7,.6-.3*.7)--(2.1-.3*.7,.3-.3*.7)--(2.1,.3)--(2.1,.6);
\end{tikzpicture}
}
\put(9,.7){$\longrightarrow$}
\put(7,1.8){\scriptsize{cutting off $h'$ and decompose it}}
\put(8,1.4){\scriptsize{ up to height $h_d$}}

\put(10.6,0){
\begin{tikzpicture}
\draw (0,0)--(2.4,0) (0,0)--(0,1.3)  (0,0)--(-.8,-.8);

\filldraw [fill=lgr,draw=black] (-.5,.3)--(-.5-.3*.7,.3-.3*.7)--(-.2-.3*.7,.3-.3*.7)--(-.2,.3)--(-.5,.3);
\filldraw [fill=lgr,draw=black] (-.2,.3)--(-.2-.3*.7,.3-.3*.7)--(.1-.3*.7,.3-.3*.7)--(.1,.3)--(-.2,.3);
\filldraw [fill=lgr,draw=black] (-.5-.3*.7,.3-.3*.7)--(-.5-.6*.7,.3-.6*.7)--(-.2-.6*.7,.3-.6*.7)--(-.2-.3*.7,.3-.3*.7)--(-.5-.3*.7,.3-.3*.7);
\filldraw [fill=lgr,draw=black] (-.2-.3*.7,.3-.3*.7)--(-.2-.6*.7,.3-.6*.7)--(.1-.6*.7,.3-.6*.7)--(.1-.3*.7,.3-.3*.7)--(-.2-.3*.7,.3-.3*.7);
\filldraw [fill=lgr,draw=black] (-.5-.6*.7,.3-.6*.7)--(-.5-.9*.7,.3-.9*.7)--(-.2-.9*.7,.3-.9*.7)--(-.2-.6*.7,.3-.6*.7)--(-.5-.6*.7,.3-.6*.7);
\filldraw [fill=lgr,draw=black] (-.2-.6*.7,.3-.6*.7)--(-.2-.9*.7,.3-.9*.7)--(.1-.9*.7,.3-.9*.7)--(.1-.6*.7,.3-.6*.7)--(-.2-.6*.7,.3-.6*.7);
\filldraw [fill=dgr,draw=black] (.1,.3)--(.1-.3*.7,.3-.3*.7)--(.1-.3*.7,-.3*.7)--(.1,0)--(.1,.3);
\filldraw [fill=dgr,draw=black] (.1-.3*.7,.3-.3*.7)--(.1-.6*.7,.3-.6*.7)--(.1-.6*.7,-.6*.7)--(.1-.3*.7,-.3*.7)--(.1-.3*.7,.3-.3*.7);
\filldraw [fill=dgr,draw=black] (.1-.6*.7,.3-.6*.7)--(.1-.9*.7,.3-.9*.7)--(.1-.9*.7,-.9*.7)--(.1-.6*.7,-.6*.7)--(.1-.6*.7,.3-.6*.7);
\filldraw [fill=gr,draw=black] (-.5-.9*.7,.3-.9*.7)rectangle(-.2-.9*.7,-.9*.7);
\filldraw [fill=gr,draw=black] (-.2-.9*.7,.3-.9*.7)rectangle(.1-.9*.7,-.9*.7);

\filldraw [fill=lgr,draw=black] (-.2,.9)--(-.2-.3*.7,.9-.3*.7)--(.1-.3*.7,.9-.3*.7)--(.1,.9)--(-.2,.9);
\filldraw [fill=lgr,draw=black] (-.5,.9)--(-.5-.3*.7,.9-.3*.7)--(-.2-.3*.7,.9-.3*.7)--(-.2,.9)--(-.5,.9);
\filldraw [fill=gr,draw=black] (-.5-.3*.7,.6-.3*.7)rectangle(-.2-.3*.7,.3-.3*.7);
\filldraw [fill=gr,draw=black] (-.2-.3*.7,.6-.3*.7)rectangle(.1-.3*.7,.3-.3*.7);
\filldraw [fill=gr,draw=black] (-.5-.3*.7,.9-.3*.7)rectangle(-.2-.3*.7,.6-.3*.7);
\filldraw [fill=gr,draw=black] (-.2-.3*.7,.9-.3*.7)rectangle(.1-.3*.7,.6-.3*.7);
\filldraw [fill=dgr,draw=black] (.1,.9)--(.1-.3*.7,.9-.3*.7)--(.1-.3*.7,.6-.3*.7)--(.1,.6)--(.1,.9);
\filldraw [fill=dgr,draw=black] (.1,.6)--(.1-.3*.7,.6-.3*.7)--(.1-.3*.7,.3-.3*.7)--(.1,.3)--(.1,.6);

\filldraw [fill=lgr,draw=black] (.6,.6)--(.6-.3*.7,.6-.3*.7)--(.9-.3*.7,.6-.3*.7)--(.9,.6)--(.6,.6);
\filldraw [fill=lgr,draw=black] (.9,.6)--(.9-.3*.7,.6-.3*.7)--(1.2-.3*.7,.6-.3*.7)--(1.2,.6)--(.9,.6);
\filldraw [fill=lgr,draw=black] (1.2,.6)--(1.2-.3*.7,.6-.3*.7)--(1.5-.3*.7,.6-.3*.7)--(1.5,.6)--(1.2,.6);
\filldraw [fill=lgr,draw=black] (1.5,.6)--(1.5-.3*.7,.6-.3*.7)--(1.8-.3*.7,.6-.3*.7)--(1.8,.6)--(1.5,.6);
\filldraw [fill=lgr,draw=black] (1.8,.3)--(1.8-.3*.7,.3-.3*.7)--(2.1-.3*.7,.3-.3*.7)--(2.1,.3)--(1.8,.3);
\filldraw [fill=gr,draw=black] (.6-.6*.7,.3-.6*.7)rectangle(.9-.6*.7,-.6*.7);
\filldraw [fill=gr,draw=black] (.6-.3*.7,.6-.3*.7)rectangle(.9-.3*.7,.3-.3*.7);
\filldraw [fill=gr,draw=black] (.9-.3*.7,.3-.3*.7)rectangle(1.2-.3*.7,-.3*.7);
\filldraw [fill=gr,draw=black] (.9-.3*.7,.6-.3*.7)rectangle(1.2-.3*.7,.3-.3*.7);
\filldraw [fill=gr,draw=black] (1.2-.3*.7,.3-.3*.7)rectangle(1.5-.3*.7,-.3*.7);
\filldraw [fill=gr,draw=black] (1.2-.3*.7,.6-.3*.7)rectangle(1.5-.3*.7,.3-.3*.7);
\filldraw [fill=gr,draw=black] (1.5-.3*.7,.3-.3*.7)rectangle(1.8-.3*.7,-.3*.7);
\filldraw [fill=gr,draw=black] (1.5-.3*.7,.6-.3*.7)rectangle(1.8-.3*.7,.3-.3*.7);
\filldraw [fill=gr,draw=black] (1.8-.3*.7,.3-.3*.7)rectangle(2.1-.3*.7,-.3*.7);
\filldraw [fill=dgr,draw=black] (2.1,.3)--(2.1-.3*.7,.3-.3*.7)--(2.1-.3*.7,-.3*.7)--(2.1,0)--(2.1,.3);
\filldraw [fill=dgr,draw=black] (1.8,.6)--(1.8-.3*.7,.6-.3*.7)--(1.8-.3*.7,.3-.3*.7)--(1.8,.3)--(1.8,.6);
\filldraw [fill=dgr,draw=black] (.9-.3*.7,.3-.3*.7)--(.9-.6*.7,.3-.6*.7)--(.9-.6*.7,-.6*.7)--(.9-.3*.7,0-.3*.7)--(.9-.3*.7,.3-.3*.7);
\filldraw [fill=lgr,draw=black] (.6-.3*.7,.3-.3*.7)--(.6-.6*.7,.3-.6*.7)--(.9-.6*.7,.3-.6*.7)--(.9-.3*.7,.3-.3*.7)--(.6-.3*.7,.3-.3*.7);

\draw[->] (1.2,1.4)--(1.2,.9);
\draw(1.25,1.65)node{$ h'$};
\end{tikzpicture}
}

\put(10.4,-.55){\scriptsize{glueing and cutting off $h''$}}

\put(7.8,-3){
\begin{tikzpicture}
\draw [->] (1.4,2)--(.9,1.5);

\draw (0,0)--(2.4,0) (0,0)--(0,1.7)  (0,0)--(-.8,-.8);
\filldraw [fill=lgr,draw=black] (0,.3)--(-.3*.7,.3-.3*.7)--(.3-.3*.7,.3-.3*.7)--(.3,.3)--(0,.3);
\filldraw [fill=lgr,draw=black] (.3,.3)--(.3-.3*.7,.3-.3*.7)--(.6-.3*.7,.3-.3*.7)--(.6,.3)--(.3,.3);
\filldraw [fill=lgr,draw=black] (-.3*.7,.3-.3*.7)--(-.6*.7,.3-.6*.7)--(.3-.6*.7,.3-.6*.7)--(.3-.3*.7,.3-.3*.7)--(-.3*.7,.3-.3*.7);
\filldraw [fill=lgr,draw=black] (.3-.3*.7,.3-.3*.7)--(.3-.6*.7,.3-.6*.7)--(.6-.6*.7,.3-.6*.7)--(.6-.3*.7,.3-.3*.7)--(.3-.3*.7,.3-.3*.7);
\filldraw [fill=lgr,draw=black] (-.6*.7,.3-.6*.7)--(-.9*.7,.3-.9*.7)--(.3-.9*.7,.3-.9*.7)--(.3-.6*.7,.3-.6*.7)--(-.6*.7,.3-.6*.7);
\filldraw [fill=lgr,draw=black] (.3-.6*.7,.3-.6*.7)--(.3-.9*.7,.3-.9*.7)--(.6-.9*.7,.3-.9*.7)--(.6-.6*.7,.3-.6*.7)--(.3-.6*.7,.3-.6*.7);

\filldraw [fill=dgr,draw=black] (.6,.3)--(.6-.3*.7,.3-.3*.7)--(.6-.3*.7,-.3*.7)--(.6,0)--(.6,.3);
\filldraw [fill=dgr,draw=black] (.6-.3*.7,.3-.3*.7)--(.6-.6*.7,.3-.6*.7)--(.6-.6*.7,-.6*.7)--(.6-.3*.7,-.3*.7)--(.6-.3*.7,.3-.3*.7);
\filldraw [fill=dgr,draw=black] (.6-.6*.7,.3-.6*.7)--(.6-.9*.7,.3-.9*.7)--(.6-.9*.7,-.9*.7)--(.6-.6*.7,-.6*.7)--(.6-.6*.7,.3-.6*.7);
\filldraw [fill=gr,draw=black] (-.9*.7,.3-.9*.7)rectangle(.3-.9*.7,-.9*.7);
\filldraw [fill=gr,draw=black] (.3-.9*.7,.3-.9*.7)rectangle(.6-.9*.7,-.9*.7);

\filldraw [fill=lgr,draw=black] (.6,.3)--(.6-.3*.7,.3-.3*.7)--(.9-.3*.7,.3-.3*.7)--(.9,.3)--(.6,.3);
\filldraw [fill=lgr,draw=black] (.9,.3)--(.9-.3*.7,.3-.3*.7)--(1.2-.3*.7,.3-.3*.7)--(1.2,.3)--(.9,.3);
\filldraw [fill=lgr,draw=black] (1.2,.3)--(1.2-.3*.7,.3-.3*.7)--(1.5-.3*.7,.3-.3*.7)--(1.5,.3)--(1.2,.3);
\filldraw [fill=lgr,draw=black] (1.5,.3)--(1.5-.3*.7,.3-.3*.7)--(1.8-.3*.7,.3-.3*.7)--(1.8,.3)--(1.5,.3);
\filldraw [fill=lgr,draw=black] (1.8,.3)--(1.8-.3*.7,.3-.3*.7)--(2.1-.3*.7,.3-.3*.7)--(2.1,.3)--(1.8,.3);
\filldraw [fill=dgr,draw=black] (2.1,.3)--(2.1-.3*.7,.3-.3*.7)--(2.1-.3*.7,-.3*.7)--(2.1,0)--(2.1,.3);
\filldraw [fill=gr,draw=black] (.9-.3*.7,.3-.3*.7)rectangle(1.2-.3*.7,-.3*.7);
\filldraw [fill=gr,draw=black] (1.2-.3*.7,.3-.3*.7)rectangle(1.5-.3*.7,-.3*.7);
\filldraw [fill=gr,draw=black] (1.5-.3*.7,.3-.3*.7)rectangle(1.8-.3*.7,-.3*.7);
\filldraw [fill=gr,draw=black] (1.8-.3*.7,.3-.3*.7)rectangle(2.1-.3*.7,-.3*.7);
\filldraw [fill=gr,draw=black] (.6-.6*.7,.3-.6*.7)rectangle(.9-.6*.7,-.6*.7);
\filldraw [fill=dgr,draw=black] (.9-.3*.7,.3-.3*.7)--(.9-.6*.7,.3-.6*.7)--(.9-.6*.7,-.6*.7)--(.9-.3*.7,0-.3*.7)--(.9-.3*.7,.3-.3*.7);
\filldraw [fill=lgr,draw=black] (.6-.3*.7,.3-.3*.7)--(.6-.6*.7,.3-.6*.7)--(.9-.6*.7,.3-.6*.7)--(.9-.3*.7,.3-.3*.7)--(.6-.3*.7,.3-.3*.7);

\filldraw [fill=lgr,draw=black] (.3,1.4)--(.3-.3*.7,1.4-.3*.7)--(.6-.3*.7,1.4-.3*.7)--(.6,1.4)--(.3,1.4);
\filldraw [fill=lgr,draw=black] (0,1.4)--(-.3*.7,1.4-.3*.7)--(.3-.3*.7,1.4-.3*.7)--(.3,1.4)--(0,1.4);
\filldraw [fill=gr,draw=black] (-.3*.7,1.1-.3*.7)rectangle(.3-.3*.7,.8-.3*.7);
\filldraw [fill=gr,draw=black] (.3-.3*.7,1.1-.3*.7)rectangle(.6-.3*.7,.8-.3*.7);
\filldraw [fill=gr,draw=black] (-.3*.7,1.4-.3*.7)rectangle(.3-.3*.7,1.1-.3*.7);
\filldraw [fill=gr,draw=black] (.3-.3*.7,1.4-.3*.7)rectangle(.6-.3*.7,1.1-.3*.7);
\filldraw [fill=dgr,draw=black] (.6,1.4)--(.6-.3*.7,1.4-.3*.7)--(.6-.3*.7,1.1-.3*.7)--(.6,1.1)--(.6,1.4);

\filldraw [fill=lgr,draw=black] (.6,1.1)--(.6-.3*.7,1.1-.3*.7)--(.9-.3*.7,1.1-.3*.7)--(.9,1.1)--(.6,1.1);
\filldraw [fill=lgr,draw=black] (.9,1.1)--(.9-.3*.7,1.1-.3*.7)--(1.2-.3*.7,1.1-.3*.7)--(1.2,1.1)--(.9,1.1);
\filldraw [fill=lgr,draw=black] (1.2,1.1)--(1.2-.3*.7,1.1-.3*.7)--(1.5-.3*.7,1.1-.3*.7)--(1.5,1.1)--(1.2,1.1);
\filldraw [fill=lgr,draw=black] (1.5,1.1)--(1.5-.3*.7,1.1-.3*.7)--(1.8-.3*.7,1.1-.3*.7)--(1.8,1.1)--(1.5,1.1);
\filldraw [fill=gr,draw=black] (.6-.3*.7,1.1-.3*.7)rectangle(.9-.3*.7,.8-.3*.7);
\filldraw [fill=gr,draw=black] (.9-.3*.7,1.1-.3*.7)rectangle(1.2-.3*.7,.8-.3*.7);
\filldraw [fill=gr,draw=black] (1.2-.3*.7,1.1-.3*.7)rectangle(1.5-.3*.7,.8-.3*.7);
\filldraw [fill=gr,draw=black] (1.5-.3*.7,1.1-.3*.7)rectangle(1.8-.3*.7,.8-.3*.7);
\filldraw [fill=dgr,draw=black] (1.8,1.1)--(1.8-.3*.7,1.1-.3*.7)--(1.8-.3*.7,.8-.3*.7)--(1.8,.8)--(1.8,1.1);

\put(2.3,.9) {$\longleftarrow h''$}
 \end{tikzpicture}
}

\put(6.8,-1.95) {$\longleftarrow $}
\put(4.8,-1.2) {\scriptsize{decompose $h''$ and glue}}

\put(2.7,-3){
\begin{tikzpicture}
\draw (0,0)--(2.5,0) (0,0)--(0,1.3)  (0,0)--(-.8,-.8);
\filldraw [fill=lgr,draw=black] (0,.3)--(-.3*.7,.3-.3*.7)--(.3-.3*.7,.3-.3*.7)--(.3,.3)--(0,.3);
\filldraw [fill=lgr,draw=black] (.3,.3)--(.3-.3*.7,.3-.3*.7)--(.6-.3*.7,.3-.3*.7)--(.6,.3)--(.3,.3);

\filldraw [fill=lgr,draw=black] (-.6*.7,.3-.6*.7)--(-.9*.7,.3-.9*.7)--(.3-.9*.7,.3-.9*.7)--(.3-.6*.7,.3-.6*.7)--(-.6*.7,.3-.6*.7);
\filldraw [fill=lgr,draw=black] (.3-.6*.7,.3-.6*.7)--(.3-.9*.7,.3-.9*.7)--(.6-.9*.7,.3-.9*.7)--(.6-.6*.7,.3-.6*.7)--(.3-.6*.7,.3-.6*.7);

\filldraw [fill=dgr,draw=black] (.6,.3)--(.6-.3*.7,.3-.3*.7)--(.6-.3*.7,-.3*.7)--(.6,0)--(.6,.3);
\filldraw [fill=dgr,draw=black] (.6-.3*.7,.3-.3*.7)--(.6-.6*.7,.3-.6*.7)--(.6-.6*.7,-.6*.7)--(.6-.3*.7,-.3*.7)--(.6-.3*.7,.3-.3*.7);
\filldraw [fill=dgr,draw=black] (.6-.6*.7,.3-.6*.7)--(.6-.9*.7,.3-.9*.7)--(.6-.9*.7,-.9*.7)--(.6-.6*.7,-.6*.7)--(.6-.6*.7,.3-.6*.7);

\filldraw [fill=gr,draw=black] (-.9*.7,.3-.9*.7)rectangle(.3-.9*.7,-.9*.7);
\filldraw [fill=gr,draw=black] (.3-.9*.7,.3-.9*.7)rectangle(.6-.9*.7,-.9*.7);

\filldraw [fill=lgr,draw=black] (.3,.9)--(.3-.3*.7,.9-.3*.7)--(.6-.3*.7,.9-.3*.7)--(.6,.9)--(.3,.9);
\filldraw [fill=lgr,draw=black] (0,.9)--(-.3*.7,.9-.3*.7)--(.3-.3*.7,.9-.3*.7)--(.3,.9)--(0,.9);
\filldraw [fill=gr,draw=black] (-.3*.7,.6-.3*.7)rectangle(.3-.3*.7,.3-.3*.7);
\filldraw [fill=gr,draw=black] (.3-.3*.7,.6-.3*.7)rectangle(.6-.3*.7,.3-.3*.7);
\filldraw [fill=gr,draw=black] (-.3*.7,.9-.3*.7)rectangle(.3-.3*.7,.6-.3*.7);
\filldraw [fill=gr,draw=black] (.3-.3*.7,.9-.3*.7)rectangle(.6-.3*.7,.6-.3*.7);
\filldraw [fill=dgr,draw=black] (.6,.9)--(.6-.3*.7,.9-.3*.7)--(.6-.3*.7,.6-.3*.7)--(.6,.6)--(.6,.9);

\filldraw [fill=lgr,draw=black] (.6,.6)--(.6-.3*.7,.6-.3*.7)--(.9-.3*.7,.6-.3*.7)--(.9,.6)--(.6,.6);
\filldraw [fill=lgr,draw=black] (.9,.6)--(.9-.3*.7,.6-.3*.7)--(1.2-.3*.7,.6-.3*.7)--(1.2,.6)--(.9,.6);
\filldraw [fill=lgr,draw=black] (1.2,.45)--(1.2-.3*.7,.45-.3*.7)--(1.5-.3*.7,.45-.3*.7)--(1.5,.45)--(1.2,.45);
\filldraw [fill=lgr,draw=black] (1.5,.45)--(1.5-.3*.7,.45-.3*.7)--(1.8-.3*.7,.45-.3*.7)--(1.8,.45)--(1.5,.45);
\filldraw [fill=lgr,draw=black] (1.8,.3)--(1.8-.3*.7,.3-.3*.7)--(2.1-.3*.7,.3-.3*.7)--(2.1,.3)--(1.8,.3);
\filldraw [fill=gr,draw=black] (.6-.3*.7,.3-.3*.7)rectangle(.9-.3*.7,-.3*.7);
\filldraw [fill=gr,draw=black] (.6-.3*.7,.6-.3*.7)rectangle(.9-.3*.7,.3-.3*.7);
\filldraw [fill=gr,draw=black] (.9-.3*.7,.3-.3*.7)rectangle(1.2-.3*.7,-.3*.7);
\filldraw [fill=gr,draw=black] (.9-.3*.7,.6-.3*.7)rectangle(1.2-.3*.7,.3-.3*.7);
\filldraw [fill=gr,draw=black] (1.2-.3*.7,.3-.3*.7)rectangle(1.5-.3*.7,-.3*.7);
\filldraw [fill=gr,draw=black] (1.2-.3*.7,.45-.3*.7)rectangle(1.5-.3*.7,.3-.3*.7);
\filldraw [fill=gr,draw=black] (1.5-.3*.7,.3-.3*.7)rectangle(1.8-.3*.7,-.3*.7);
\filldraw [fill=gr,draw=black] (1.5-.3*.7,.45-.3*.7)rectangle(1.8-.3*.7,.3-.3*.7);
\filldraw [fill=gr,draw=black] (1.8-.3*.7,.3-.3*.7)rectangle(2.1-.3*.7,-.3*.7);
\filldraw [fill=dgr,draw=black] (2.1,.3)--(2.1-.3*.7,.3-.3*.7)--(2.1-.3*.7,-.3*.7)--(2.1,0)--(2.1,.3);
\filldraw [fill=dgr,draw=black] (1.8,.45)--(1.8-.3*.7,.45-.3*.7)--(1.8-.3*.7,.3-.3*.7)--(1.8,.3)--(1.8,.45);
\filldraw [fill=dgr,draw=black] (1.2,.45)--(1.2-.3*.7,.45-.3*.7)--(1.2-.3*.7,.6-.3*.7)--(1.2,.6)--(1.2,.45);

\filldraw [fill=lgr,draw=black] (-.3*.7,.45-.3*.7)--(-.6*.7,.45-.6*.7)--(.3-.6*.7,.45-.6*.7)--(.3-.3*.7,.45-.3*.7)--(-.3*.7,.45-.3*.7);
\filldraw [fill=lgr,draw=black] (.3-.3*.7,.45-.3*.7)--(.3-.6*.7,.45-.6*.7)--(.6-.6*.7,.45-.6*.7)--(.6-.3*.7,.45-.3*.7)--(.3-.3*.7,.45-.3*.7);
\filldraw [fill=gr,draw=black] (-.6*.7,.45-.6*.7)rectangle(.3-.6*.7,.3-.6*.7);
\filldraw [fill=gr,draw=black] (.3-.6*.7,.45-.6*.7)rectangle(.6-.6*.7,.3-.6*.7);
\filldraw [fill=dgr,draw=black] (.6-.3*.7,.3-.3*.7)--(.6-.6*.7,.3-.6*.7)--(.6-.6*.7,.45-.6*.7)--(.6-.3*.7,.45-.3*.7)--(.6-.3*.7,.3-.3*.7);

\filldraw [fill=gr,draw=black] (.6-.6*.7,.3-.6*.7)rectangle(.9-.6*.7,-.6*.7);
\filldraw [fill=dgr,draw=black] (.9-.3*.7,.3-.3*.7)--(.9-.6*.7,.3-.6*.7)--(.9-.6*.7,-.6*.7)--(.9-.3*.7,0-.3*.7)--(.9-.3*.7,.3-.3*.7);
\filldraw [fill=lgr,draw=black] (.6-.3*.7,.3-.3*.7)--(.6-.6*.7,.3-.6*.7)--(.9-.6*.7,.3-.6*.7)--(.9-.3*.7,.3-.3*.7)--(.6-.3*.7,.3-.3*.7);
\end{tikzpicture}
}

}
\end{picture}

The algorithm leads to $h=t^9*s^4+ \frac{1}{2} s^5 + \frac{1}{2}s^3+s^1$, however without shifting the boxes we get $h=t^9*t^4*s^0 + \frac{1}{2}s^6+ \frac{1}{2}s^5+s^1$.

The example shows that shifting the boxes is not really necessary to get a decomposition in extremal points and again that this decomposition is not unique. But that it may be important to get a totally ordered chain of extremal points.
\end{exa}

Now we know that $\Ex(d)$ is a generating system of \Hc$(d)$. It remains to show that the points are extremal:

\begin{lemma}
Let $v\in\Ex(d)$. Then $v$ is not in the convex hull of the remaining points in $\Ex(d)$.
\end{lemma}

\begin{proof}
Let $n \in \mathbb{N}$ \,and\,$d \geq1$ be fixed, $v=\sum_{i \geq 0} q_iv^i$ with $v^i \in \Ex(d)$ and $q_i\geq 0.$
Since $v_0^i= 1$ for all $v^i \in \Ex(d)$ we get $\sum q_i =1.$
\begin{itemize}
\item [(i)] If $v=s^\ell$ with $\ell \leq d$ each $v^i$ belongs to $\Ex(\ell).$\\
            In the cases $\ell \leq n-1$ or $\ell=n\!\cdot \!k-1$ for any $k$, the only extremal point in $\Ex(\ell)$ with 
            nontrivial $(\ell+1)$-th component is $s^\ell$.\\
            Otherwise $v_\ell=s_\ell = \lfloor\frac{\ell}{n}\rfloor+1 >1$ \,and the only extremal points of length \linebreak $\ell+1$             with nontrivial $(\ell+1)$-th component except $s^\ell$ are $t^\ell$ and $t^\ell*w$ for \linebreak
            $w\in\Ex(\ell-3-2r)$ with $\ell=nm+r,\; r \in \{0,\dots, n-1\}$ and their entry is always 1.
            Assuming $v^i \neq s^\ell$ we get the contradiction $v_\ell=(\sum q_iv^i)_\ell \leq 1.$
\item [(ii)] Let $v=t^\ell*w$ with $\ell=nm+r, r \neq n-1$ and  $w\in\Ex(\ell-2r-3) \cup\{(0)\}$.
             We know $v_{nm-1}=0$ \,and\,$v_0=v_{nm}=\!...=v_{nm+r+1}=1,$ therefore all $v^i$ in the decomposition are of the              form $t^\ell*w^i$ with  $w^i\in\Ex(\ell-2r-3)\cup \{(0)\}$.\\
             As the $*$-operator is a shifted addition we get\\
             $$v=t^\ell*w=\sum q_i(t^\ell*w^i)=(\sum q_i) t^\ell* \sum q_iw^i=t^\ell*\sum q_iw^i.$$
             This gives $w=\sum q_iw^i $ and by induction over $d$ we are done.\qedhere
          
\end{itemize}
\end{proof}

\newpage
\section*{Appendix}

Flowchart of the \textbf{ ALGORITHM:}\\

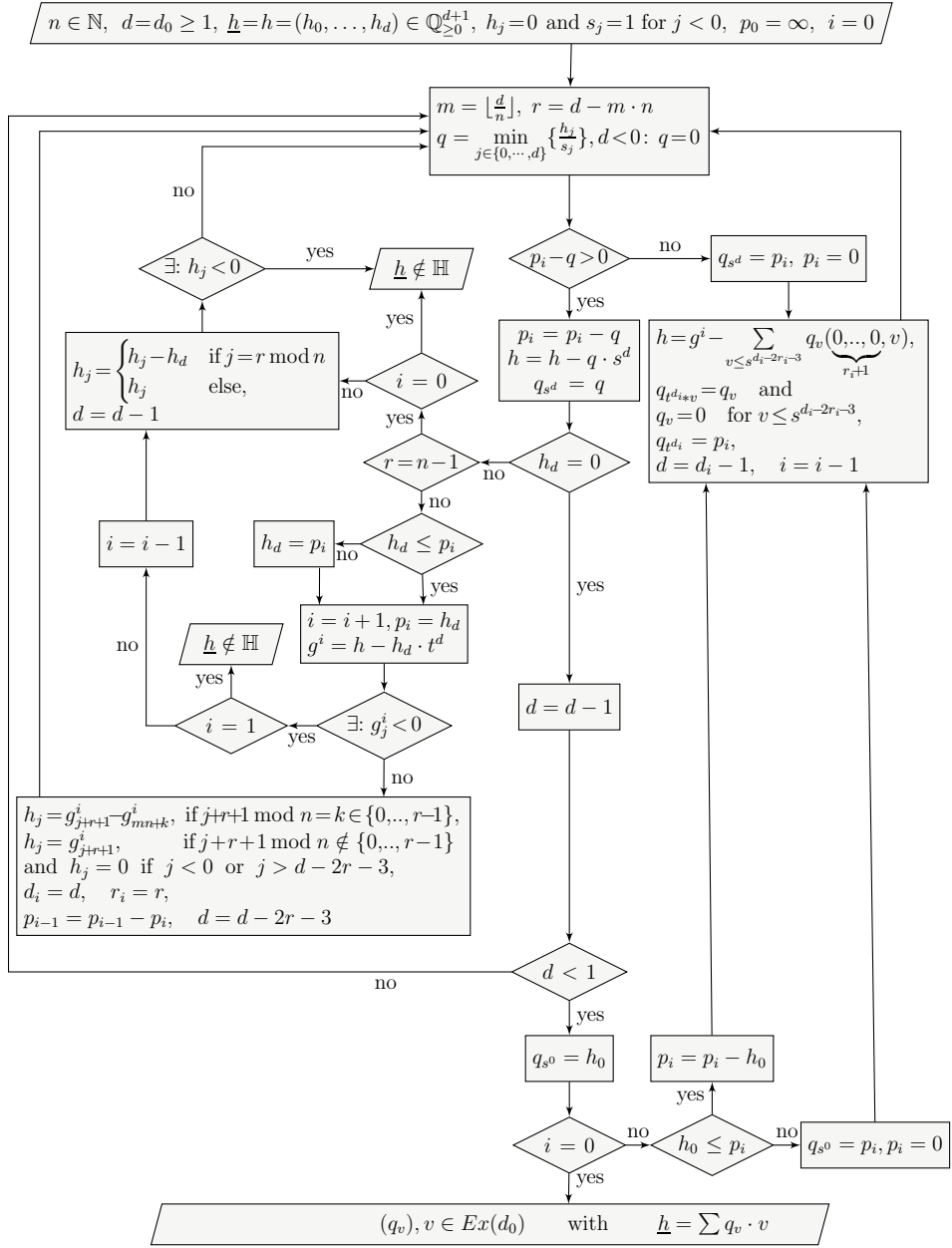
\bw{

\tikzstyle{decision} = [shape aspect=2,diamond, draw, fill=gr!7, 
    text width=4em, text centered, inner sep=0pt]
\tikzstyle{block1} = [rectangle, draw, fill=gr!7, 
    minimum size=9mm, text centered]
\tikzstyle{block2} = [rectangle, draw, fill=gr!7, 
    text width=6em, text centered, minimum height=3em]
\tikzstyle{block3} = [rectangle, draw, fill=gr!7, 
    text width=52mm, text centered,minimum height=3em]
\tikzstyle{line} = [draw, -latex']

\tikzstyle{para1} = [draw,trapezium,trapezium left angle=75,trapezium right angle=-75,fill=gr!7,minimum size=8mm]
\tikzstyle{para2} = [draw,trapezium,trapezium left angle=75,trapezium right angle=-75,fill=gr!7,minimum size=8mm]

\vspace{2em}

\begin{tikzpicture}[node distance = 1cm, auto,scale=0.68, transform shape]
    \node [para1] (eingabe) {$ n \in \mathbb{N},\;\, d\!=\!d_0\geq 1,\; \underline{h}\!=\!h\!=\!(h_0,\dots,h_d) \in \mathbb{Q}_{\geq0}^{d+1},\;  h_j\!=\!0 \:\, $and$ \;s_j\!=\!1\;$for$\; j<0, \;\, p_0=\infty,\;\, i=0$};
    \node [block3, below right of=eingabe,node distance=30mm] (q) {\begin{minipage}{55mm}$m = \lfloor \frac{d}{n}\rfloor, \;r=d-m \cdot n \\q=\underset{j \in \{0,\cdots,d\}}{\min}\{\frac{h_j}{s_j}\}, d\!<\!0\!: \, q\!=\!0$\\ \end{minipage}};
    \node [decision, below of=q, node distance=25mm] (diff) {$p_i\!-\!q\!>\!0$};
    \node [block1, right of=diff, node distance=43mm] (qgross) {$q_{s^d}=p_i,\; p_i=0$};
    \node [block3, below of=qgross, node distance=28mm] (iklein) {\begin{minipage}{54mm}$h\!=\!g^i\!-\!\sum\limits_{v \leq s^{d_i\!-\!2r_i\!-\!3}}q_{v}(\underbrace{0,\!..,0}_{r_i\!+\!1},v), \\  q_{t^{d_i}\! \ast v}\!=\!q_v \quad \text{and}\\ q_{v}\!=\!0  \quad \text{for}\; v\!\leq \!s^{d_i\!-\!2r_i\!-\!3},\\ q_{t^{d_i}}=p_i,\\  d=d_i-1,\quad  i=i-1$\end{minipage}};
    \node [block2, below of=diff, node distance=20mm] (neux) {$p_i=p_i-q$ $ h=h-q\cdot s^d$ $ q_{s^d}=q$};
    \node [decision, below of=neux, node distance=20mm] (max) {$h_d=0$};
    \node [decision, left of=max, node distance=29mm] (odd) {$r \!= \!n\!-\!1$};
    \node [decision, above of=odd, node distance=16mm] (io) {$i=0$};
    \node [block1, left of=io, node distance=43mm] (dmin2){\begin{minipage}{50.8mm}$h_j\!=\!\begin{cases} \! h_j\!-\!h_d & \text{if} \, j\!=\!r \,\text{mod}\, n \\ \! h_j & \text{else},\end{cases} \\ d=d-1 $\end{minipage}};
    \node [decision,above of=dmin2, node distance=22mm](xjklein0){$\exists \!\!: h_j\!<\!0$};
    \node [para2, above of=io, node distance=22mm] (xnotcone) {$\underline{h} \notin \mathbb{H}$};
    \node [decision, below of=odd, node distance=16mm] (kleinpi){$h_d\leq p_i$};
    \node [block1, left of=kleinpi, node distance= 25mm] (gleichpi) {$h_d=p_i$};
    \node [block1, below right of=gleichpi, node distance= 25mm] (dmin3) {\begin{minipage}{30mm}$i=i+1, p_i=h_d\\g^i=h-h_d\cdot t^d$ \end{minipage}};
    \node [decision, below of=dmin3, node distance= 18mm] (yjklein0) {$\exists\!\!: g_j^i\!<\!0$};
    \node [decision, left of=yjklein0, node distance= 30mm] (igleich1) {$i=1$};
    \node [para2, above of= igleich1, node distance=16mm] (nichthcone) {$\underline{h} \notin \mathbb{H}$};
    \node [block1, left of= gleichpi, node distance= 29mm] (imin1) {$i=i-1$};
    \node [block1, below left of=yjklein0, node distance= 39mm] (x)  {\begin{minipage}{86.2mm}$ h_j\!= \! g_{j\!+\!r\!+\!1}^i\!-\!g_{mn\!+\!k}^i,\; \text{if}\, j\!+\!r\!+\!1 \,\text{mod}\; n \!=\! k \!\in \! \{0,\!..,r\!-\!1\}, \\ h_j\!=g_{j\!+\!r\!+\!1}^i,\hspace{11.5mm} \text{if}\, j\!+\!r\!+\!1 \,\text{mod}\; n  \notin  \{0,\!..,r\!-\!1\}  \\ \text{and}\;\;h_j=0\;\; \text{if} \;\;j < 0 \;\;\text{or}\;\; j > d-2r-3, \\ d_i=d, \quad r_i = r,\\p_{i-1}=p_{i-1}-p_i, \quad d=d-2r-3$ \end{minipage}};
    \node [block1, below of=max,node distance=48mm] (dim) {$d=d-1$};
    \node [decision, below of=dim, node distance=52mm] (dklein) {$d< 1$};
    \node [block1, below of=dklein, node distance=17mm] (x1) {$q_{s^0}=h_0$};
    \node [decision, below of=x1, node distance=17mm] (i0) {$i=0$};
    \node [decision, right of=i0, node distance=28mm] (x1kleinp1) {$h_0 \leq p_i$};
    \node [block1, right of=x1kleinp1, node distance=32mm] (pi0) {$q_{s^0}=p_i, p_i=0$};
    \node [block1, above of=x1kleinp1, node distance=17mm] (piminx1) {$p_i=p_i-h_0$};
    \node [para1, below left of=i0, node distance=22mm] (qi) {$\qquad \qquad \qquad \qquad(q_{v}),v \in Ex(d_0)$\qquad with \qquad$\underline{h}=\sum q_{v}\cdot v$};

    \draw[-latex']  ($(eingabe)+(21.6mm,-4mm)$)-- (q);
    \path [line] (q) --  (diff);
    \path [line] (diff) -- node  {yes} (neux);
    \path [line] (diff) -- node {no} (qgross);
    \path [line] (qgross) -- (iklein);
    \draw[-latex'] ($(iklein)+(22mm,15.9mm)$)|-(q);
    \path [line] (neux) -- (max);
    \path [line] (max)-- node {no} (odd);
    \path [line] (odd) -- node {yes} (io);
    \path [line] (odd) -- node {no} (kleinpi);
    \path [line] (kleinpi) -- node {no} (gleichpi);
    \draw[-latex'] ($(gleichpi)+(5mm,-4.7mm)$) -- ($(dmin3)+(-12.7mm,5.6mm)$);
    \draw[-latex'] (kleinpi)-- node {yes} ($(dmin3)+(7.7mm,5.6mm)$);
    \draw[-latex'] ($(x)+(-40mm,13.7mm)$)|-(q);
    \path [line] (dmin3)--(yjklein0);
    \path [line] (yjklein0)--node {no}($(x)+(27.5mm,13.5mm)$);
    \path [line] (yjklein0)-- node {yes} (igleich1);
    \path [line] (igleich1)-- node {yes} (nichthcone);
    \draw[-latex'] (igleich1)-|node [near end]{no} (imin1);
    \path [line] (imin1)-- ($(dmin2)+(-11mm,-9.6mm)$);
    \path [line] (io)-- node {yes} (xnotcone);
    \path [line] (io)-- node {no} (dmin2);
    \path [line] (dmin2)--(xjklein0);
    \path [line] (xjklein0)-- node {yes} (xnotcone);
    \draw[-latex'] (xjklein0)|- node [near start] {no}($(q)+(-27.5mm,-3mm)$);
    \path [line] (max)-- node {yes} (dim);
    \draw[-latex'] (dklein)-- node [near start] {no} ++(-110mm,0) |- ($(q)+(-27.5mm,3mm)$);
    \path [line] (dim) -- (dklein);
    \path [line] (dklein) -- node {yes} (x1);
    \path [line] (x1) -- (i0);
    \path [line] (i0) -- node {no} (x1kleinp1);
    \path [line] (x1kleinp1) -- node {no} (pi0);
    \path [line] (x1kleinp1) -- node {yes} (piminx1);
    \draw[-latex'] (piminx1)--($ (iklein)+(-16.2mm,-16mm)$);
    \draw[-latex'] (pi0)--($(iklein)+(15.2mm,-16mm)$);
    \draw[-latex'] (i0) -- node [near start]{yes} ($(qi)+(15.4mm,4mm)$);
\end{tikzpicture}}
\vspace{1cm}

The $v'$s here are always elements of $\Ex(d), \,d$ being large enough.

\pagebreak

\section*{Acknowledgements}
We would like to thank the organizers of the workshop P.R.A.G.Mat.I.C, in particular Prof. Alfio Ragusa and Guiseppe Zappal\`a for providing an excellent environment for collaboration and research in Catania, Italy in the summer 2011.
We also would like to thank Prof. Mats Boij and Prof. Ralf Fr\"oberg as well as Dr. Alexander Engstr\"om for their outstanding lectures providing interesting problems, their support, suggestions and ideas.\\

\nocite{*}
\bibliography{hcone.v1}

\end{document}